\documentclass[a4paper]{amsart}
\usepackage[utf8]{inputenc}
\usepackage[T1]{fontenc}
\usepackage{lmodern}
\usepackage{amssymb}
\usepackage{tikz}
\usetikzlibrary{matrix,calc,arrows,decorations.pathmorphing}
\tikzset{node distance=2cm, auto}
\tikzset{cd/.style=matrix of math nodes,row sep=2em,column sep=2em, text height=1.5ex, text depth=0.5ex}
\tikzset{cdar/.style=->,auto}

\tikzset{triar/.style={anchor=mid,->}}
\tikzset{tridar/.style={anchor=mid,double,double equal sign distance,-implies}}
\tikzset{narrowfill/.style={inner sep=0pt, fill=white}}

\tikzset{dar/.style={double,double equal sign distance,-implies}}
\tikzset{mid/.style={anchor=mid}} 

\usepackage{nicefrac,mathtools,enumitem}
\usepackage{microtype}
\usepackage[pdftitle={Crossed products for actions of crossed modules on C*-algebras},
  pdfauthor={Alcides Buss, Ralf Meyer},
  pdfsubject={Mathematics}
]{hyperref}
\usepackage[lite]{amsrefs}
\newcommand*{\MRref}[2]{ \href{http://www.ams.org/mathscinet-getitem?mr=#1}{MR #1}}
\newcommand*{\arxiv}[1]{\href{http://www.arxiv.org/abs/#1}{arXiv: #1}}
\renewcommand{\PrintDOI}[1]{\href{http://dx.doi.org/\detokenize{#1}}{doi: \detokenize{#1}}%
  \IfEmptyBibField{pages}{, (to appear in print)}{}}

\setlist[enumerate,1]{label=\textup{(\arabic*)}}
\setlist[enumerate,2]{label=\textup{(\alph*)}}

\usepackage{xcolor}

\numberwithin{equation}{section}
\theoremstyle{plain}
\newtheorem{theorem}[equation]{Theorem}
\newtheorem{lemma}[equation]{Lemma}
\newtheorem{proposition}[equation]{Proposition}

\theoremstyle{definition}
\newtheorem{definition}[equation]{Definition}
\theoremstyle{remark}
\newtheorem{example}[equation]{Example}
\newtheorem{remark}[equation]{Remark}

\DeclareMathOperator{\coker}{coker}
\DeclareMathOperator{\Aut}{Aut}
\DeclareMathOperator{\Ad}{Ad}
\DeclareMathOperator{\cspn}{\overline{span}}

\newcommand*{\nb}{\nobreakdash}
\newcommand*{\Star}{$^*$\nobreakdash-}
\newcommand*{\blank}{\textup{\textvisiblespace}}

\newcommand*{\Lag}{\mathfrak g}
\newcommand*{\Lah}{\mathfrak h}
\newcommand*{\Lak}{\mathfrak k}

\newcommand*{\Q}{\mathbb Q}
\newcommand*{\Z}{\mathbb Z}
\newcommand*{\R}{\mathbb R}
\newcommand*{\N}{\mathbb N}
\newcommand*{\Torus}{\mathbb T}
\newcommand*{\Comp}{\mathbb K}
\newcommand*{\Bound}{\mathbb B}

\newcommand*{\sbe}{\subseteq} 

\newcommand*{\Hilm}[1][E]{\mathcal #1}

\newcommand*{\Cst}{\textup C^*}
\newcommand*{\cstar}{\texorpdfstring{$C^*$\nobreakdash-\hspace{0pt}}{C*-}}
\newcommand*{\Cont}{\textup C}
\newcommand*{\Contc}{\textup C_\textup c}
\newcommand*{\contc}{\Contc}
\newcommand*{\contz}{\Cont_0}

\newcommand*{\Mult}{\mathcal M}
\newcommand*{\Rotc}[1]{\Cst(\Torus_#1)}

\newcommand*{\K}{\textup K}

\newcommand*{\ima}{\textup i}
\newcommand*{\un}{\textup u}
\renewcommand*{\u}{\upsilon}
\newcommand*{\Id}{\textup{Id}}

\newcommand*{\E}{\mathcal E} 
\newcommand*{\U}{\mathcal U} 
\newcommand*{\A}{\mathcal A} 
\newcommand*{\B}{\mathcal B} 
\newcommand*{\I}{\mathcal I} 
\newcommand*{\cm}{\mathcal C}
\newcommand*{\acm}{c}
\newcommand*{\tcm}{\partial}

\newcommand*{\defeq}{\mathrel{\vcentcolon=}}
\newcommand*{\into}{\rightarrowtail}
\newcommand*{\onto}{\twoheadrightarrow}
\newcommand*{\congto}{\xrightarrow\sim}

\newcommand*{\norm}[1]{\lVert#1\rVert}
\newcommand*{\cl}[1]{\overline{#1}}
\newcommand*{\conj}[1]{\overline{#1}}

\newcommand*{\Csttwocat}{\mathfrak C^*(2)}
\newcommand*{\Corrcat}{\mathfrak{Corr}(2)}
\newcommand*{\Corr}{\mathfrak{Corr}}

\begin{document}
\title[Crossed products for crossed modules]{Crossed products for actions of crossed modules on C*-algebras}

\author{Alcides Buss}
\email{alcides@mtm.ufsc.br}
\address{Departamento de Matem\'atica\\
 Universidade Federal de Santa Catarina\\
 88.040-900 Florian\'opolis-SC\\
 Brazil}

\author{Ralf Meyer}
\email{rmeyer2@uni-goettingen.de}
\address{Mathematisches Institut\\
 Georg-August-Universit\"at G\"ottingen\\
 Bunsenstra\ss e 3--5\\
 37073 G\"ottingen\\
 Germany}

\begin{abstract}
  We decompose the crossed product functor for actions of crossed
  modules of locally compact groups on \cstar{}algebras into more
  elementary constructions: taking crossed products by group actions
  and fibres in \cstar{}algebras over topological spaces.  For this,
  we extend Takesaki--Takai duality to Abelian crossed modules;
  describe the crossed product for an extension of crossed modules;
  show that equivalent crossed modules have equivalent categories of
  actions on \cstar{}algebras; and show that certain crossed modules
  are automatically equivalent to Abelian crossed modules.
\end{abstract}


\thanks{Supported by the German Research Foundation (Deutsche
  Forschungsgemeinschaft (DFG)) grant ME\,3248/1 and by CNPq
  (Ciências sem Fronteira) -- Brazil.}

\maketitle

\section{Introduction}
\label{sec:introduction}

A crossed module of locally compact groups~\(\cm\) consists of two
locally compact groups \(H\) and~\(G\) with a continuous group
homomorphism \(\tcm\colon H\to G\) and a continuous conjugation action
\(\acm\colon G\to\Aut(H)\) such that
\[
\tcm(\acm_g(h))=g\tcm(h)g^{-1},\qquad
\acm_{\tcm(h)}(k) = hkh^{-1}
\]
for all \(g\in G\), \(h,k\in H\).  Strict actions of crossed modules
on \(\Cst\)\nb-algebras and crossed products for such actions are
defined in~\cite{Buss-Meyer-Zhu:Non-Hausdorff_symmetries}.  Here we
are going to factorise this crossed product functor into more
elementary operations, namely, taking crossed products for actions
of locally compact \emph{groups} and taking fibres in
\(\Cont_0(X)\)-\(\Cst\)-algebras (\cite{Nilsen:Bundles}).

We mostly work with the more flexible notion of action by
correspondences introduced in~\cite{Buss-Meyer-Zhu:Higher_twisted}.
By \cite{Buss-Meyer-Zhu:Higher_twisted}*{Theorem 5.3}, such actions
are Morita--Rieffel equivalent to ordinary strict actions, that is,
actions by automorphisms.  This requires, however, to stabilise the
\(\Cst\)\nb-algebras involved; and certain induced actions that we
need are easier to describe as actions by correspondences.  Here we
define the \(2\)\nb-category of crossed module actions by correspondences
precisely, making explicit some hints
in~\cite{Buss-Meyer-Zhu:Higher_twisted}.  Then we translate the
definition of this \(2\)\nb-category into the language of Fell bundles;
this extends results in~\cite{Buss-Meyer-Zhu:Higher_twisted} from
locally compact groups to crossed modules.  We define saturated Fell
bundles over crossed modules and correspondences between Fell
bundles over crossed modules so that they are equivalent to actions
by correspondences and transformations between such actions.
Correspondences between Fell bundles contain Morita--Rieffel
equivalences of Fell bundles and representations of Fell bundles as
special cases.  The crossed product for a crossed module action by
correspondences is defined by a universal property for
representations.

Let \(\cm_i=(G_i,H_i,\tcm_i,\acm_i)\) be crossed modules of locally
compact groups.  We call a diagram \(\cm_1\to\cm_2\to\cm_3\) of
homomorphisms of crossed modules a (strict) extension if the
resulting diagrams \(G_1\to G_2\to G_3\) and \(H_1\to H_2\to H_3\)
are extensions of locally compact groups in the usual sense.
If~\(\cm_2\) acts on a \(\Cst\)\nb-algebra~\(A\), then \(\cm_1\)
also acts on \(A\) by restriction of the \(\cm_2\)\nb-action.  We
show that
\[
A\rtimes \cm_2 \cong (A\rtimes\cm_1)\rtimes\cm_3
\]
for a certain induced action of~\(\cm_3\) by correspondences on
\(A\rtimes\cm_1\).

If~\(\cm_i\) are ordinary groups~\(G_i\) viewed as crossed modules,
then \(G_1\) is a closed normal subgroup of~\(G_2\) with
quotient~\(G_3=G_2/G_1\).  It is well-known that \(A\rtimes G_1\)
carries a Green twisted action of~\((G_2,G_1)\) such that
\((A\rtimes G_1)\rtimes (G_2,G_1)\cong A\rtimes G_2\).  Our theorem
says that such a Green twisted action may be turned into an action
of~\(G_3\) by correspondences with the same crossed product.  Our
proof, in fact, generalises this idea to cover general extensions of
crossed modules.  Another special case is
\cite{Buss-Meyer-Zhu:Non-Hausdorff_symmetries}*{Theorem 1}, which
says that \((A\rtimes H)\rtimes \cm\cong A\rtimes G\) if
\(\cm=(G,H,\tcm,\acm)\) is a crossed module and~\(A\) is a
\(G\)\nb-\(\Cst\)-algebra.

Moreover, we generalise Takesaki--Takai duality to Abelian crossed
modules.  We call a crossed module Abelian if the group~\(G\) is
Abelian and the action~\(\acm\) is trivial (forcing~\(H\) to be
Abelian).  Thus Abelian crossed modules are just continuous
homomorphisms \(\tcm\colon H\to G\) between Abelian locally compact
groups.  The Pontryagin dual~\(\hat{\cm}\) is the crossed module
\(\hat{\tcm}\colon \hat{G}\to\hat{H}\) given by the transpose
of~\(\tcm\).  Our duality theorem says that actions of~\(\cm\) are
equivalent to actions of the \emph{arrow groupoid}
\(\hat{H}\rtimes\hat{G}\) associated to~\(\cm\); this is the
transformation groupoid for the action of~\(\hat{G}\) on~\(\hat{H}\)
where \(\hat{g}\in\hat{G}\) acts by right translations by
\(\hat{\tcm}(\hat{g})\).

Our duality maps an action of~\(\cm\) to its crossed product by~\(G\),
equipped with the dual action of~\(\hat{G}\) and a canonical
\(\Cont_0(\hat{H})\)-\(\Cst\)-algebra structure that comes from the
original action of~\(H\).  The inverse equivalence takes the crossed
product by~\(\hat{G}\) and extends the dual \(G\)\nb-action on it to
an action of~\(\cm\), using the original
\(\Cont_0(\hat{H})\)-\(\Cst\)-algebra structure.  Thus our duality
result merely enriches the usual Takesaki--Takai duality by
translating the action of~\(H\) for a crossed module action into a
\(\Cont_0(\hat{H})\)-\(\Cst\)-algebra structure on the crossed
product, and vice versa.

In this setting, the crossed product for crossed module actions is
equivalent to the functor of taking the fibre at \(1\in\hat{H}\) for
an action of \(\hat{H}\rtimes\hat{G}\).  Thus crossed products for
Abelian crossed modules may be computed in two steps: first take a
crossed product by an action of the Abelian group~\(G\), then take a
fibre for a \(\Cont_0(\hat{H})\)-\(\Cst\)-algebra structure.

Now we describe our decomposition of crossed products for a general
crossed module \(\cm=(G,H,\tcm,\acm)\).  Let \(G_1\) be the trivial
group and \(H_1\defeq \ker\tcm\).  Since \(\ker\tcm\) is Abelian,
the trivial maps \(\tcm_1\) and~\(\acm_1\) provide a crossed
module~\(\cm_1=(G_1,H_1,\tcm_1,\acm_1)\).  This fits into an
extension \(\cm_1\into \cm\onto \cm_2\),
\(\cm_2=(G_2,H_2,\tcm_2,\acm_2)\), where \(G_2\defeq G\),
\(H_2\defeq H/H_1\) and \(\tcm_2\) and~\(\acm_2\) are the canonical
induced maps.

Next, let \(G_3\defeq \cl{\tcm_2(H_2)}\subseteq G_2\), \(H_3\defeq
H_2\) and let \(\tcm_3\) and~\(\acm_3\) be the restrictions of
\(\tcm_2\) and~\(\acm_2\).  The crossed module~\(\cm_3=(G_3,H_3,\tcm_3,\acm_3)\) has the
feature that~\(\tcm_3\) is injective with dense range; we call such
crossed modules \emph{thin}.

Since~\(G_3\) is a closed normal subgroup of \(G\), \(G_4\defeq
G/G_3\) is a locally compact group.  Let \(H_4=0\), \(\tcm_4\)
and~\(\acm_4\) be trivial.  This gives another strict extension of
crossed modules \(\cm_3\to \cm_2\to \cm_4\).  Two applications of
Theorem~\ref{the:partial_crossed} give
\[
A\rtimes\cm \cong (A\rtimes\cm_1) \rtimes \cm_2\cong ((A\rtimes\cm_1) \rtimes \cm_3)\rtimes\cm_4.
\]
Thus it remains to study crossed products by crossed modules of the
special forms \(\cm_1\), \(\cm_3\) and~\(\cm_4\), where \(G_1=0\),
\(\cm_3\) is thin, and \(H_4=0\).

Since \(G_1=0\), \(\cm_1\) is a very particular Abelian crossed
module.  Here our duality says that actions of~\(\cm_1\) on~\(A\)
are equivalent to \(\Cont_0(\widehat{H_1})\)-\(\Cst\)-algebra
structures on~\(A\), where~\(\widehat{H_1}\) denotes the dual group
of~\(H_1\).  The crossed product with~\(\cm_1\) is the fibre at
\(1\in\widehat{H_1}\) for the corresponding
\(\Cont_0(\widehat{H_1})\)-\(\Cst\)-algebra structure.

Since \(H_4=0\), an action of~\(\cm_4\) is equivalent to an action
of the group~\(G_4\); crossed products also have the usual meaning.

To understand crossed products for the thin crossed module~\(\cm_3\),
we replace~\(\cm_3\) by a simpler but equivalent crossed module.
Equivalent crossed modules have equivalent categories of actions on
\(\Cst\)\nb-algebras, and the equivalence preserves both the
underlying \(\Cst\)\nb-algebra of the action and the crossed products.
Many thin crossed modules are equivalent to Abelian crossed modules.
We prove this for all thin crossed modules of Lie groups and provide
both a sufficient criterion and a counterexample for thin crossed
modules of locally compact groups.

If the thin crossed module~\(\cm_3\) is equivalent to an Abelian
crossed module~\(\cm_5\), then we may turn a \(\cm_3\)\nb-action
on~\(A\) into a \(\cm_5\)\nb-action on~\(A\) with an isomorphic
crossed product.  By our duality theory, the crossed product
by~\(\cm_5\) is the fibre at \(1\in\widehat{H_5}\) for a canonical
\(\Cont_0(\widehat{H_5})\)-\(\Cst\)-algebra structure on \(A\rtimes
G_5\).

Assuming that~\(\cm_3\) satisfies the mild condition to make it
Abelian, we thus decompose the crossed product functor for our
original crossed module~\(\cm\) into four more elementary steps:
taking fibres in \(\Cont_0(X)\)-\(\Cst\)-algebras twice and taking
crossed products by ordinary groups twice.

\section{Crossed module actions by correspondences}
\label{sec:Fell_cm}

A \emph{correspondence} between two \cstar{}algebras \(A\) and~\(B\)
is a Hilbert \(B\)\nb-module~\(\E\) with a nondegenerate
\Star{}homomorphism \(A\to \Bound(\E)\).  Correspondences are the
arrows of a weak \(2\)\nb-category \(\Corrcat\), with \cstar{}algebras
as objects and isomorphisms of correspondences as \(2\)\nb-arrows
(see~\cite{Buss-Meyer-Zhu:Higher_twisted}).

Let \(\cm=(H,G,\tcm,\acm)\) be a crossed module of locally compact
groups.  We may turn~\(\cm\) into a strict \(2\)\nb-group with group
of arrows~\(G\) and \(2\)\nb-arrow space~\(G\times H\),
where~\((g,h)\) gives a \(2\)\nb-arrow \(g\Rightarrow g\tcm(h)\).

Following \cite{Buss-Meyer-Zhu:Higher_twisted}*{Section 4}, we
define an \emph{action of~\(\cm\) by correspondences} as a morphism
\(\cm\to\Corrcat\) in the sense of
\cite{Buss-Meyer-Zhu:Higher_twisted}*{Definition 4.1}, with
continuity conditions added as in
\cite{Buss-Meyer-Zhu:Higher_twisted}*{Section 4.1}; we define
\emph{transformations} between such actions as in
\cite{Buss-Meyer-Zhu:Higher_twisted}*{Section 4.2}, again with extra
continuity requirements; and we define \emph{modifications} between
such transformations as in
\cite{Buss-Meyer-Zhu:Higher_twisted}*{Section 4.3}.  This defines a
\(2\)\nb-category \(\Corr(\cm)\).

The definitions in~\cite{Buss-Meyer-Zhu:Higher_twisted} for actions
of general strict \(2\)\nb-categories may be simplified
because~\(\cm\) is a strict \(2\)\nb-group.  Analogous
simplifications are already discussed in detail
in~\cite{Buss-Meyer-Zhu:Higher_twisted} for weak actions by
automorphisms, that is, morphisms to the
\(2\)\nb-category~\(\Csttwocat\)
(see~\cite{Buss-Meyer-Zhu:Higher_twisted} for the definition
of~\(\Csttwocat\)).  For this reason, we merely state the simplified
definitions without proving that they are equivalent to those in
\cite{Buss-Meyer-Zhu:Higher_twisted}*{Section 4}.

\begin{definition}
  \label{def:act_corr}
  Let \(\cm=(H,G,\tcm,\acm)\) be a crossed module of locally compact
  groups. An \emph{action of~\(\cm\) by correspondences} consists of
  \begin{itemize}
  \item a \(\Cst\)\nb-algebra~\(A\);
  \item correspondences \(\alpha_g\colon A\to A\) for all \(g\in G\)
    with \(\alpha_1=A\) the identity correspondence;
  \item isomorphisms of correspondences \(\omega_{g_1,g_2}\colon
    \alpha_{g_2}\otimes_A \alpha_{g_1}\to \alpha_{g_1g_2}\) for
    \(g_1,g_2\in G\), where \(\omega_{1,g}\) and \(\omega_{g,1}\)
    are the canonical isomorphisms \(\alpha_g\otimes_A A\cong
    \alpha_g\) and \(A\otimes_A\alpha_g\cong \alpha_g\);
  \item isomorphisms of correspondences \(\eta_h\colon A \to
    \alpha_{\tcm(h)}\) for all \(h\in H\), where \(\eta_1=\Id_A\);
  \item a \(\Cont_0(G)\)-linear correspondence~\(\alpha\) from
    \(\Cont_0(G,A)\) to itself with fibres~\(\alpha_g\); more
    explicitly, this means a space of continuous sections
    \(\alpha\subseteq\prod_{g\in G} \alpha_g\) such that pointwise
    products of elements in~\(\alpha\) with elements of
    \(\Cont_0(G,A)\) on the left or right are again in~\(\alpha\),
    pointwise inner products of elements in~\(\alpha\) are in
    \(\Cont_0(G,A)\), and the projections \(\alpha\to\alpha_g\) are
    all surjective;
  \end{itemize}
  these must satisfy the following conditions:
  \begin{enumerate}
  \item \label{def:act_corr1} the following diagram of isomorphisms
    commutes for all \(g_1,g_2,g_3\in G\):
    \[
    \begin{tikzpicture}
      \matrix(m)[cd,column sep=8em]{
        \alpha_{g_3}\otimes_A \alpha_{g_2}\otimes_A \alpha_{g_1}&
        \alpha_{g_3}\otimes_A \alpha_{g_1g_2}\\
        \alpha_{g_2g_3}\otimes_A \alpha_{g_1}&
        \alpha_{g_1g_2g_3}\\
      };
      \draw[cdar] (m-1-1) -- node {\(\Id_{\alpha_{g_3}} \otimes_A \omega_{g_1,g_2}\)} (m-1-2);
      \draw[cdar] (m-1-1) -- node[swap] {\(\omega_{g_2,g_3} \otimes_A \Id_{\alpha_{g_1}}\)} (m-2-1);
      \draw[cdar] (m-1-2) -- node {\(\omega_{g_1g_2,g_3}\)} (m-2-2);
      \draw[cdar] (m-2-1) -- node {\(\omega_{g_1,g_2g_3}\)} (m-2-2);
    \end{tikzpicture}
    \]
  \item \label{def:act_corr2} the following diagram of isomorphisms
    commutes for all \(h_1,h_2\in H\):
    \[
    \begin{tikzpicture}
      \matrix(m)[cd,column sep=6em]{
        A\otimes_A A&
        A\\
        \alpha_{\tcm(h_2)}\otimes_A \alpha_{\tcm(h_1)}&
        \alpha_{\tcm(h_1h_2)}\\
      };
      \draw[cdar] (m-1-1) -- node {\footnotesize can} (m-1-2);
      \draw[cdar] (m-1-1) -- node[swap] {\(\eta_{h_2} \otimes_A \eta_{h_1}\)} (m-2-1);
      \draw[cdar] (m-1-2) -- node {\(\eta_{h_1h_2}\)} (m-2-2);
      \draw[cdar] (m-2-1) -- node {\(\omega_{\tcm(h_1),\tcm(h_2)}\)} (m-2-2);
    \end{tikzpicture}
    \]
  \item \label{def:act_corr3} the following diagram of isomorphisms
    commutes for all \(h\in H\), \(g\in G\):
    \begin{equation}
      \label{eq:CoherenceOmegasWithEtas}
      \begin{tikzpicture}[yscale=1.2,xscale=2,baseline=(current bounding box.west)]
        \node (m-0-1) at (0,2) {\(A\otimes_A \alpha_g\)};
        \node (m-0-2) at (1,2) {\(\alpha_g\)};
        \node (m-0-3) at (2,2) {\(\alpha_g\otimes_A A\)};
        \node (m-1-1) at (0,1) {\(\alpha_{\tcm(h)} \otimes_A \alpha_g\)};
        \node (m-1-3) at (2,1) {\(\alpha_g \otimes_A \alpha_{\tcm(\acm_gh)}\)};
        \node (m-2-1) at (0,0) {\(\alpha_{g\tcm(h)}\)};
        \node (m-2-3) at (2,0) {\(\alpha_{\tcm(\acm_gh)g}\)};
        \draw[cdar] (m-0-1) -- node {\footnotesize can} (m-0-2);
        \draw[cdar] (m-0-3) -- node[swap] {\footnotesize can} (m-0-2);
        \draw[cdar] (m-0-1) -- node[swap] {\(\eta_h\otimes_A\Id_{\alpha_g}\)} (m-1-1);
        \draw[cdar] (m-0-3) -- node {\(\Id_{\alpha_g}\otimes_A\eta_{\acm_g(h)}\)} (m-1-3);
        \draw[cdar] (m-1-1) -- node[swap] {\(\omega_{g,\tcm(h)}\)} (m-2-1);
        \draw[cdar] (m-1-3) -- node {\(\omega_{\tcm(\acm_gh),g}\)} (m-2-3);
        \draw[double,double equal sign distance] (m-2-1) -- (m-2-3);
      \end{tikzpicture}
    \end{equation}

  \item \label{def:act_corr4} fibrewise application of
    \(\omega_{g_2,g_1}\) gives an isomorphism
    \[
    \omega\colon \pi_2^*\alpha\otimes_{\Cont_0(G\times G,A)}
    \pi_1^*\alpha \to \mu^*\alpha,
    \]
    where \(\pi_1\) and~\(\pi_2\) are the two coordinate projections
    \(G\times G\to G\) and \(\mu\) is the multiplication map
    \(G\times G\to G\); notice that the fibres of these two Hilbert
    modules over \(\Cont_0(G\times G,A)\) at \((g_1,g_2)\in G\times
    G\) are \(\alpha_{g_2}\otimes_A\alpha_{g_1}\)
    and~\(\alpha_{g_1g_2}\), respectively;

  \item \label{def:act_corr5} fibrewise application of
    \((\eta_h)_{h\in H}\) gives an isomorphism \(\eta\colon
    \Cont_0(H,A)\to \tcm^*\alpha\).
  \end{enumerate}
\end{definition}

Since the maps \(\omega_{g_1,g_2}\) and~\(\eta_h\) are isomorphisms,
the continuity conditions \ref{def:act_corr4}
and~\ref{def:act_corr5} hold if
\((\omega_{g_1,g_2})_{g_1,g_2\in G}\) maps \(\pi_2^*\alpha\otimes_A
\pi_1^*\alpha\subseteq \prod_{g_1,g_2\in G}
\alpha_{g_2}\otimes_A\alpha_{g_1}\) into \(\mu^*\alpha\subseteq
\prod_{g_1,g_2\in G} \alpha_{g_1g_2}\) and~\((\eta_h)_{h\in H}\)
maps \(\Cont_0(H,A) \subseteq \prod_{h\in H} A\) into
\(\tcm^*\alpha\subseteq \prod_{h\in H} \alpha_{\tcm(h)}\); these
maps are automatically unitary (isometric and surjective).

The \emph{trivial} \(\cm\)\nb-action on~\(B\) is given by
\(\beta_g=B\) for \(g\in G\), \(\omega_{g_2,g_1}=\Id_B\) for
\(g_1,g_2\in G\), \(\eta_h=\Id_B\) for \(h\in H\), and
\(\beta=\Cont_0(G,B)\).

\begin{definition}
  \label{def:transformation}
  Let \((A,\alpha,\omega^A,\eta^A)\) and \((B,\beta,\omega^B,\eta^B)\)
  be \(\cm\)\nb-actions by correspondences.  A
  \emph{\(\cm\)\nb-equivariant correspondence} or
  \emph{transformation} between them consists of
  \begin{itemize}
  \item a correspondence~\(\Hilm\) from~\(A\) to~\(B\), and
  \item isomorphisms of correspondences \(\chi_g\colon
    \Hilm\otimes_B \beta_g \to \alpha_g\otimes_A \Hilm\) with
    \(\chi_1=1\),
  \end{itemize}
  such that
  \begin{enumerate}
  \item \label{def:transformation1} for all \(h\in H\), the following
    diagram commutes:
    \begin{equation}
      \label{eq:transformation_H_natural}
      \begin{tikzpicture}[yscale=1.2,xscale=2,baseline=(current bounding box.west)]
        \node (m-0-1) at (0,2) {\(A\otimes_A \Hilm\)};
        \node (m-0-2) at (1,2) {\(\Hilm\)};
        \node (m-0-3) at (2,2) {\(\Hilm\otimes_B B\)};
        \node (m-1-1) at (0,1) {\(\alpha_{\tcm(h)} \otimes_A \Hilm\)};
        \node (m-1-3) at (2,1) {\(\Hilm \otimes_B \beta_{\tcm(h)}\)};
        \draw[cdar] (m-0-1) -- node {\footnotesize can} (m-0-2);
        \draw[cdar] (m-0-3) -- node[swap] {\footnotesize can} (m-0-2);
        \draw[cdar] (m-0-1) -- node[swap] {\(\eta^A_h\otimes_A\Id_{\Hilm}\)} (m-1-1);
        \draw[cdar] (m-0-3) -- node {\(\Id_{\Hilm}\otimes_B\eta^B_h\)} (m-1-3);
        \draw[cdar] (m-1-3) -- node {\(\chi_{\tcm(h)}\)} (m-1-1);
      \end{tikzpicture}
    \end{equation}

  \item \label{def:transformation2} for all \(g_1,g_2\in G\), the
    following diagram commutes:
    \begin{equation}
      \label{eq:twisted_conjugacy_coherence}
      \begin{tikzpicture}[yscale=1.2,xscale=3,baseline=(current bounding box.west)]
        \node (m-1-1) at (0,2)
        {\(\Hilm \otimes_B \beta_{g_2} \otimes_B \beta_{g_1}\)};
        \node (m-1-2) at (2,2)
        {\(\alpha_{g_2} \otimes_A \Hilm \otimes_B \beta_{g_1}\)};
        \node (m-2-1) at (0,1)
        {\(\Hilm \otimes_B \beta_{g_1g_2}\)};
        \node (m-2-2) at (2,1)
        {\(\alpha_{g_2}\otimes_A \alpha_{g_1} \otimes_A \Hilm\)};
        \node (m-3-2) at (1,0)
        {\(\alpha_{g_1g_2}\otimes_A \Hilm\)};

        \draw[cdar] (m-1-1) -- node
        {\(\chi_{g_2} \otimes_B \Id_{\beta_{g_1}}\)} (m-1-2);
        \draw[cdar] (m-1-2) -- node
        {\(\Id_{\alpha_{g_2}}\otimes_A \chi_{g_1}\)} (m-2-2);
        \draw[cdar] (m-1-1) -- node[swap]
        {\(\Id_{\Hilm} \otimes_B \omega^B_{g_1,g_2}\)} (m-2-1);
        \draw[cdar] (m-2-2) -- node
        {\(\omega^A_{g_1,g_2}\otimes_A \Id_{\Hilm}\)} (m-3-2);
        \draw[cdar] (m-2-1) -- node[swap] {\(\chi_{g_1g_2}\)} (m-3-2);
      \end{tikzpicture}
    \end{equation}

  \item \label{def:transformation3} pointwise application
    of~\(\chi_g\) gives a \(\Cont_0(G)\)-linear isomorphism
    \(\chi\colon \Hilm\otimes_B \beta \to \alpha\otimes_A \Hilm\).
  \end{enumerate}

  A transformation is called an (equivariant Morita--Rieffel)
  \emph{equivalence} if~\(\Hilm\) is an equivalence, that is, the
  left \(A\)\nb-action is given by an isomorphism
  \(A\cong\Comp(\Hilm)\).

  A transformation is called a (covariant) \emph{representation} of
  \((A,\alpha)\) on \(B\) if the \(\cm\)\nb-action on~\(B\) is
  trivial and~\(\Hilm\) is the correspondence associated to a
  nondegenerate \Star{}homomorphism \(\pi\colon A\to \Mult(B)\). In
  this case, we also write \(\pi\colon (A,\alpha)\to B\) or simply
  \(\pi\colon A\to B\) to denote the representation.
\end{definition}

The continuity condition~\ref{def:transformation3} is equivalent to
\((\chi_g)_{g\in G}\) mapping \(\Hilm\otimes_B \beta\subseteq
\prod_{g\in G} \Hilm\otimes_B\beta_g\) into \(\alpha\otimes_A
\Hilm\subseteq \prod_{g\in G} \alpha_g\otimes_A \Hilm\); this map is
automatically unitary because it is fibrewise unitary.

\begin{definition}
  \label{def:modification}
  A \emph{modification} between two transformations
  \((\Hilm,\chi_g)\) and \((\Hilm',\chi'_g)\) from~\(A\)
  to~\(B\) is a unitary \(W\colon \Hilm\to \Hilm'\) such that for all
  \(g\in G\) the following diagram commutes:
  \begin{equation}
    \label{eq:modification}
    \begin{tikzpicture}[yscale=1.2,xscale=2.4,baseline=(current bounding box.west)]
      \node (m-0-1) at (0,1) {\(\Hilm\otimes_B \beta_g\)};
      \node (m-0-2) at (1,1) {\(\alpha_g \otimes_A\Hilm\)};
      \node (m-1-1) at (0,0) {\(\Hilm'\otimes_B \beta_g\)};
      \node (m-1-2) at (1,0) {\(\alpha_g \otimes_A\Hilm'\)};
      \draw[cdar] (m-0-1) -- node {\(\chi_g\)} (m-0-2);
      \draw[cdar] (m-1-1) -- node[swap] {\(\chi'_g\)} (m-1-2);
      \draw[cdar] (m-0-1) -- node[swap] {\(W\otimes_B\Id_{\beta_g}\)} (m-1-1);
      \draw[cdar] (m-0-2) -- node {\(\Id_{\alpha_g}\otimes_A W\)} (m-1-2);
    \end{tikzpicture}
  \end{equation}
\end{definition}

The notion of representation above leads to a definition of crossed products:

\begin{definition}
  \label{def:crossed_product}
  A \emph{crossed product} for a \(\cm\)\nb-action
  \((A,\alpha,\omega,\eta)\) by correspondences is a
  \cstar{}algebra~\(B\) with a representation \(\pi^\un\colon  A\to B\) that is universal in the sense that any other representation
  \(A\to C\) factors uniquely as \(f\circ\pi^\un\) for a morphism
  (nondegenerate \Star{}homomorphism) \(f\colon B\to\Mult(C)\).
\end{definition}

If a crossed product exists, then its universal property determines
it uniquely up to canonical isomorphism because an isomorphism in
the morphism category of \(\Cst\)\nb-algebras must be a
\Star{}isomorphism in the usual sense.  We will construct crossed
products later using cross-sectional \cstar{}algebras of Fell
bundles.

The actions, transformations, and modifications defined above are
the objects, arrows and \(2\)\nb-arrows of a weak \(2\)\nb-category
(that is, bicategory) \(\Corr(\cm)\) with invertible \(2\)\nb-arrows.  This
statement contains the following assertions.  Given two
\(\cm\)\nb-actions \(x_i=(A_i,\alpha_i,\omega_i,\eta_i)\) for
\(i=1,2\), the transformations \(x_1\to x_2\) (as objects) and the modifications (as morphisms)
between such transformations form a groupoid \(\Corr_\cm(x_1,x_2)\);
here the composition of modifications is just the composition of
unitary operators.  Given three \(\cm\)\nb-actions \(x_1\), \(x_2\)
and~\(x_3\), there is a composition bifunctor
\[
\Corr_\cm(x_2,x_3)\times \Corr_\cm(x_1,x_2) \to \Corr_\cm(x_1,x_3);
\]
the composite of two transformations \((\Hilm_1,\chi_1)\) from
\(x_1\) to~\(x_2\) and \((\Hilm_2,\chi_2)\) from~\(x_2\)
to~\(x_3\) is the transformation from~\(x_1\) to~\(x_3\) consisting
of \(\Hilm_1\otimes_{A_2} \Hilm_2\) and
\[
\Hilm_1\otimes_{A_2} \Hilm_2 \otimes_{A_3} \alpha_{3g}
\xrightarrow{\Id_{\Hilm_1}\otimes_{A_2} \chi_{2g}}
\Hilm_1\otimes_{A_2} \alpha_{2g} \otimes_{A_1} \Hilm_2
\xrightarrow{\chi_{1g}\otimes_{A_1} \Id_{\Hilm_2}}
\alpha_{1g} \otimes_{A_1} \Hilm_1\otimes_{A_2} \Hilm_2.
\]
This composition is associative and unital up to canonical
isomorphisms with suitable coherence properties.  To see the
associators, notice that \((\Hilm_1\otimes_{A_2} \Hilm_2)
\otimes_{A_3} \Hilm_3\) and \(\Hilm_1\otimes_{A_2} (\Hilm_2
\otimes_{A_3} \Hilm_3)\) are not identical but merely canonically
isomorphic.  These canonical isomorphisms are the associators.  They
are canonical enough that being careful about them would lead to
more confusion than leaving them out.  The identity arrow on~\(A\)
is~\(A\) with the canonical isomorphisms \(\chi_g\colon
A\otimes_A\alpha_g\cong \alpha_g\cong \alpha_g\otimes_A A\).  The
composite of another arrow with such an identity arrow is
canonically isomorphic to that arrow using the canonical
isomorphisms \(A\otimes_A\Hilm\cong\Hilm\) and \(\Hilm\otimes_B
B\cong\Hilm\).  These are the unit transformations.  The coherence
conditions for a weak \(2\)\nb-category listed in
\cite{Buss-Meyer-Zhu:Higher_twisted}*{Section 2.2.1} are trivially
satisfied.

\subsection{Translation to Fell bundles}
\label{sec:translate_Fell}

For a locally compact group, it is shown
in~\cite{Buss-Meyer-Zhu:Higher_twisted} that actions by
correspondences are equivalent to saturated Fell bundles.  When we
reinterpret everything in terms of Fell bundles, transformations
become correspondences between Fell bundles, modifications become
isomorphisms of such correspondences, and representations become
representations of Fell bundles in the usual sense.  Thus the
cross-sectional \cstar{}algebra of the associated Fell bundle has
the correct universal property for a crossed product.  We want to
extend all these results to crossed module actions.

Most of the work is already done
in~\cite{Buss-Meyer-Zhu:Higher_twisted}.  Let
\(\cm=(G,H,\tcm,\acm)\) be a crossed module of locally compact
groups and let \((A,\alpha,\omega,\eta)\) be an action of~\(\cm\) as
in Definition~\ref{def:act_corr}.  Forgetting~\(\eta\), the data
\((A,\alpha,\omega)\) is a continuous action by correspondences of
the locally compact group~\(G\) (as defined in
\cite{Buss-Meyer-Zhu:Higher_twisted} or as in our
Definition~\ref{def:act_corr} for~\(G\) viewed as a crossed module).
Results in~\cite{Buss-Meyer-Zhu:Higher_twisted} about locally
compact groups show that the data \((A,\alpha,\omega)\) is
equivalent to a saturated Fell bundle over~\(G\).  The map~\(\eta\)
gives us some extra data that describes how the group~\(H\) acts.  A
transformation between actions of~\(\cm\) is the same as a
transformation between the resulting actions of~\(G\) that satisfies
an additional compatibility condition with the \(H\)\nb-actions.
And modifications for \(G\)- and \(\cm\)\nb-actions are just the
same.

It remains to translate everything in \(\Corr(\cm)\) related to the
group~\(H\) to the language of Fell bundles.  For this we first need some notation.

Let~\(\A\) be a Fell bundle over~\(G\) with fibres~\(\A_g\) at
\(g\in G\).  A \emph{multiplier of order \(g\in G\)} of~\(\A\) is a
pair \(\mu=(L,R)\) (left and right multipliers) of maps
\(L,R\colon\A\to\A\) such that \(L(\A_f)\subseteq \A_{gf}\) and
\(R(\A_f)\subseteq \A_{fg}\) for all \(f\in G\), and \(aL(b)=R(a)b\)
for all \(a,b\in \A\) (see
\cite{Doran-Fell:Representations_2}*{VIII.2.14}).  We write
\(\Mult(\A)_g\) for the set of multipliers of order~\(g\).  We
usually write \(\mu\cdot a=L(a)\) and \(a\cdot\mu=R(a)\).

The maps \(L\) and~\(R\) must be fibrewise linear and bounded.  The
adjoint of~\(\mu\) is defined by \(\mu^*\cdot a=(a^*\cdot\mu)^*\)
and \(a\cdot\mu^*=(\mu\cdot a^*)^*\), and~\(\mu\) is called unitary
if \(\mu^*\mu=\mu\mu^*=1\) (the unit of \(\Mult(\A_1)\)).  The set
\(\Mult(\A)=\bigcup_{g\in G} \Mult(\A)_g\) of all multipliers
of~\(\A\) is a Fell bundle over~\(G\) viewed as a discrete group,
called the \emph{multiplier Fell bundle} of~\(\A\).  We
endow~\(\Mult(\A)\) with the \emph{strict topology}: a
net~\((\mu_i)\) in~\(\Mult(\A)\) converges strictly to \(\mu\in
\Mult(\A)\) if and only if \(\mu_i\cdot a\to \mu\cdot a\) and
\(a\cdot\mu_i\to a\cdot \mu\) in~\(\A\) for all \(a\in \A\).
Let \(\U\Mult(\A)\) be the
group of unitary multipliers of~\(\A\) of arbitrary order.

We may view the fibres~\(\A_g\) as Hilbert bimodules over~\(\A_1\)
using the multiplication in the Fell bundle and the inner products
\(\langle x,y\rangle_{\A_1}= x^*y\) on the right and \(_{\A_1}\langle
x,y\rangle= xy^*\) on the left.  Taking these operations
fibrewise makes the space \(\Gamma_0(\A)\) of continuous sections
of~\(\A\) vanishing at infinity a Hilbert bimodule over
\(\Cont_0(G,\A_1)\) because the multiplication and involution in the
Fell bundle are continuous.

\begin{lemma}
  \label{lem:MultipliersOfFellBundlesAndBimodules}
  Let \(\A\) be a saturated Fell bundle.  Then \(\Mult(\A)_g\) is
  isomorphic to the space of adjointable operators \(\A_1\to\A_g\).
  The space of strictly continuous sections of \(\Mult(\A)\) is
  isomorphic to the space of adjointable operators
  \(\Cont_0(G,\A_1)\to\Gamma_0(\A)\), that is, the multiplier
  Hilbert bimodule of \(\Gamma_0(\A)\) \textup(as defined in
  \cite{Echterhoff-Kaliszewski-Quigg-Raeburn:Categorical}*{Chapter
    1.2}\textup).
\end{lemma}

\begin{proof}
  A multiplier~\(\mu\) of~\(\A\) of order~\(g\) restricts to an
  adjointable map \(\A_1\to\A_g\), \(a\mapsto \mu\cdot a\), with
  adjoint \(b\mapsto \mu^*\cdot b\).  Furthermore, since
  \(\A_1\cdot\A_{g_2}=\A_{g_2}\) for all \(g_2\in G\), an
  adjointable map \(\A_1\to\A_g\) extends uniquely to a multiplier
  of~\(\A\).  For the last statement, we must show that a section
  \((\mu_g)_{g\in G}\) of \(\Mult(\A)\) is strictly continuous if
  and only if pointwise application of \(\mu_g\) and~\(\mu_g^*\)
  gives well-defined maps \(\Cont_0(G,\A_1)\leftrightarrow
  \Gamma_0(\A)\).  The existence of a map \(\Cont_0(G,\A_1)\to
  \Gamma_0(\A)\) is equivalent to the continuity of \(\mu_g\cdot a\)
  for all \(a\in\A_1\), which is equivalent to the continuity of
  \(\mu_g\cdot a\) for all \(a\in\A\) because \(\A=\A_1\cdot\A\).
  Since the pointwise product maps for a saturated Fell bundle
  satisfy
  \[
  \Cont_0(G,\A_1)\cdot\Gamma_0(\A) = \Gamma_0(\A)
  \quad\text{and}\quad
  \Gamma_0(\A)\cdot\Gamma_0(\A^*) = \Cont_0(G,\A_1),
  \]
  (with \(\A^*_g\defeq\A_{g^{-1}}\)), the continuity of
  \(\mu_g^*\cdot a_g\) for all \((a_g)\in\Gamma_0(\A)\) is equivalent
  to the continuity of \(\mu_g^*\cdot a_g\) for all
  \((a_g)\in\Cont_0(G,\A_1)\).  This is in turn equivalent to the
  continuity of \(a\cdot \mu_g\) for all \(a\in\A_1\) or to the
  continuity of \(a\cdot \mu_g\) for all \(a\in\A\).
\end{proof}

\begin{definition}
  \label{def:FellBundle2Group}
  Let \(\cm=(G,H,\tcm,\acm)\) be a crossed module of locally compact
  groups.  A \emph{Fell bundle} over~\(\cm\) is a Fell bundle
  \(\A=(\A_g)_{g\in G}\) over~\(G\) together with a strictly
  continuous group homomorphism \(\u\colon H\to \U\Mult(\A)\), such
  that
  \begin{enumerate}
  \item \label{def:FellBundle2Group1} \(\u_h\) has order \(\tcm(h)\) for all \(h\in H\), and
  \item \label{def:FellBundle2Group2} \(a\cdot
    \u_h=\u_{\acm_g(h)}\cdot a\) for all \(a\in \A_g\) and \(h\in
    H\).
  \end{enumerate}
  A Fell bundle over~\(\cm\) is \emph{saturated} if it is saturated
  as a Fell bundle over~\(G\).
\end{definition}

\begin{theorem}
  \label{theo:ActionCorrespondences=FellBundles}
  Actions of~\(\cm\) by correspondences are
  equivalent to saturated Fell bundles over~\(\cm\).
\end{theorem}

\begin{proof}
  Let \((A,\alpha,\omega,\eta)\) be an action of~\(\cm\) by
  correspondences as in Definition~\ref{def:act_corr}.

  The data \((A, \alpha, \omega)\) subject to the conditions
  \ref{def:act_corr1} and~\ref{def:act_corr4} in
  Definition~\ref{def:act_corr} are an action of the locally compact
  group~\(G\) by correspondences.  We have simplified this compared
  to the definition in~\cite{Buss-Meyer-Zhu:Higher_twisted} by
  requiring the unit transformation \(A\to\alpha_1\) to be the
  identity.  It is shown as in the proof of
  \cite{Buss-Meyer-Zhu:Higher_twisted}*{Lemma 3.7} that any action
  of~\(G\) by correspondences is equivalent to one with this extra
  property.  Thus our data \((A,\alpha,\omega)\) is equivalent to a
  saturated Fell bundle over~\(G\) by
  \cite{Buss-Meyer-Zhu:Higher_twisted}*{Theorem 3.17}.  Its fibres
  are \(\A_g=\alpha_{g^{-1}}\); the multiplication
  \(\A_{g_1}\times\A_{g_2}\to \A_{g_1g_2}\) is \(a\cdot b\defeq
  \omega_{g_2^{-1},g_1^{-1}}(a\otimes b)\); the involution is the
  unique one for which the inner product on \(\A_g=\alpha_{g^{-1}}\)
  has the expected form: \(\langle x,y\rangle = x^*\cdot y\).

  An isomorphism of right Hilbert \(A\)\nb-modules
  \(\eta_{h^{-1}}\colon A\to \alpha_{\tcm(h^{-1})}= \A_{\tcm(h)}\)
  is equivalent to a unitary multiplier~\(\u_h\) of~\(\A\) of
  order~\(\tcm(h)\) by
  Lemma~\ref{lem:MultipliersOfFellBundlesAndBimodules}.  We claim
  that the conditions in Definition~\ref{def:FellBundle2Group} for
  this map \(\u\colon H\to\U\Mult(\A)\) are equivalent to the
  conditions \ref{def:act_corr2}, \ref{def:act_corr3}
  and~\ref{def:act_corr5} in Definition~\ref{def:act_corr}.

  The condition \(\omega_{\tcm(h_1),\tcm(h_2)}
  (\eta_{h_2}\otimes_A\eta_{h_1}) = \eta_{h_1h_2}\) is equivalent
  to~\(\u\) being a group homomorphism.  The
  diagram~\eqref{eq:CoherenceOmegasWithEtas} commutes if and only if
  \(\u_h^*\cdot a_1\cdot a_g\cdot a_1' = a_1\cdot a_g\cdot
  \u_{\acm_g^{-1}(h)}^*\cdot a_1'\) for all \(a_1,a_1'\in \A_1\),
  \(a_g\in \A_g\), \(g\in G\), \(h\in H\); here we use
  \(\A_1\cdot\A_g\cdot\A_1=\A_g\) and that \(\u_h^*=\u_{h^{-1}}\).
  Letting \(a_1\) and~\(a_1'\) run through approximate units, we get
  \(\u_h^*\cdot a_g = a_g\cdot\u_{\acm_g^{-1}(h)}^*\) for all \(h\in H\),
  \(g\in G\), \(a_g\in\A_g\), which is equivalent to
  condition~\ref{def:FellBundle2Group2} in
  Definition~\ref{def:FellBundle2Group}.

  The equivalence of the continuity conditions in Definitions
  \ref{def:act_corr} and~\ref{def:FellBundle2Group} follows as in
  the proof of the second statement in
  Lemma~\ref{lem:MultipliersOfFellBundlesAndBimodules}.
\end{proof}

Since actions by automorphisms may be viewed as actions by
correspondences,
Theorem~\ref{theo:ActionCorrespondences=FellBundles} implies that
weak actions by automorphisms as considered in
\cite{Buss-Meyer-Zhu:Higher_twisted}*{Section~4.1.1} give rise to
saturated Fell bundles.  We make this more explicit for strict
actions:

\begin{example}
  \label{exa:StrictToFellBundles}
  Let \((\alpha,u)\) be an action of \(\cm=(G,H,\tcm,\acm)\) by
  \Star{}automorphisms on a \cstar{}algebra~\(A\) (as defined in
  \cite{Buss-Meyer-Zhu:Non-Hausdorff_symmetries}*{Definition~3.1}).
  That is, \(\alpha\colon G\to \Aut(A)\) is a (strongly continuous)
  action of~\(G\) on~\(A\) by \Star{}automorphisms, \(u\colon H\to
  \U\Mult(A)\) is a strictly continuous group homomorphism, and
  \begin{enumerate}
  \item \(\alpha_{\tcm(h)}(a)=u_hau_h^*\) for all \(a\in A\) and \(h\in H\); and
  \item \(\alpha_g(u_h)=u_{\acm_g(h)}\) for all \(g\in G\) and \(h\in H\).
  \end{enumerate}
  Let \(\A=A\times_{\alpha} G\) be the semidirect product Fell
  bundle over~\(G\) for the action~\(\alpha\) (see
  \cite{Doran-Fell:Representations_2}*{VIII.4}).  Its operations are
  given by
  \[
  (a,f)\cdot(b,g)=(a\alpha_f(b),fg)
  \quad\text{and}\quad
  (a,g)^*=(\alpha_{g^{-1}}(a^*),g^{-1})
  \]
  for all \(a,b\in A\) and \(f,g\in G\).  The multiplier Fell bundle
  \(\Mult(\A)\) is isomorphic to the semidirect product Fell bundle
  \(\Mult(A)\times_\alpha G\), where~\(\alpha\) is tacitly extended
  to a \(G\)\nb-action on~\(\Mult(A)\).  Therefore, the formulas
  \[
  \u_h\cdot (a,g)\defeq (a u_h^*,\tcm(h)g)
  \quad\text{and}\quad
  (a,g)\cdot \u_h\defeq  (au_{\acm_g(h)}^*,g\tcm(h))
  \]
  for \(h\in H\), \(g\in G\), \(a\in A\) define a unitary
  multiplier~\(\u_h\) of order~\(\tcm(h)\) of~\(\A\), where
  \(\u_h=(u_h^*,\tcm(h))\).  The pair~\((\A,\u)\) is a (saturated)
  Fell bundle over~\(\cm\) in the above sense.
\end{example}

Next we translate transformations between \(\cm\)\nb-actions into
correspondences between Fell bundles over~\(\cm\).  Let
\((\A,\u^\A)\) and \((\B,\u^\B)\) be Fell bundles over~\(\cm\).  The
following definition is an extension of
\cite{Buss-Meyer-Zhu:Higher_twisted}*{Definition 3.21} from groups
to \(2\)\nb-groups (that is, crossed modules).  It should also be
compared with the notion of equivalence between Fell bundles over
groupoids appearing in \cites{Yamagami:IdealStructure,
  Muhly:BundlesGroupoids, Muhly-Williams:Equivalence.FellBundles}.

\begin{definition}
  \label{def:corr_Fell_bundle}
  A \emph{\(\cm\)\nb-equivariant correspondence} from \((\A,\u^\A)\)
  and \((\B,\u^\B)\) is a continuous Banach bundle
  \(\Hilm=(\Hilm_g)_{g\in G}\) over~\(G\) together with
  \begin{itemize}
  \item a continuous multiplication \(\A\times\Hilm\to\Hilm\) that
    maps \(\A_{g_1}\times\Hilm_{g_2}\) to~\(\Hilm_{g_1g_2}\);
  \item a continuous multiplication \(\Hilm\times\B\to\Hilm\) that
    maps \(\Hilm_{g_1}\times\B_{g_2}\) to~\(\Hilm_{g_1g_2}\);
  \item a continuous inner product
    \(\langle\blank,\blank\rangle\colon \Hilm\times\Hilm\to\B\) that
    maps \(\Hilm_{g_1}\times\Hilm_{g_2}\) to~\(\B_{g_1^{-1}g_2}\);
  \end{itemize}
  these must satisfy
  \begin{enumerate}
  \item associativity \(x\cdot (y\cdot z) = (x\cdot y)\cdot z\) for
    \((x,y,z)\) in \(\A\times \A\times\Hilm\), \(\Hilm\times \B\times
    \B\), and \(\A\times \Hilm\times \B\);
  \item \(\A_1\cdot\Hilm_g=\Hilm_g=\Hilm_g\cdot\B_1\) for all \(g\in
    G\);
  \item \(\xi_2\mapsto \langle\xi_1,\xi_2\rangle\) is fibrewise
    linear for all \(\xi_1\in\Hilm\) and
    \(\langle\xi_1,\xi_2\rangle^* = \langle\xi_2,\xi_1\rangle\) for
    all \(\xi_1,\xi_2\in\Hilm\);
  \item \(\langle \xi_1,\xi_2\cdot b_2\rangle = \langle
    \xi_1,\xi_2\rangle b_2\) for all \(b_2\in\B\),
    \(\xi_1,\xi_2\in\Hilm\);
  \item \(\langle\xi,\xi\rangle\ge0\) in~\(\B_1\) for all
    \(\xi\in\Hilm\), and \(\norm{\xi}^2=
    \norm{\langle\xi,\xi\rangle}\);
  \item \(\langle a\xi_1,\xi_2\rangle = \langle
    \xi_1,a^*\xi_2\rangle\) for all \(a\in\A\),
    \(\xi_1,\xi_2\in\Hilm\);
  \item \label{def:corr_Fell_bundle_central}
    \(\u_{\acm_g(h)}^\A\cdot\xi = \xi\cdot\u_h^\B\) for \(h\in H\),
    \(g\in G\), \(\xi\in\Hilm_g\).
  \end{enumerate}
  An \emph{isomorphism of correspondences} is a homeomorphism
  \(\Hilm\to\Hilm'\) that is compatible with the left and right
  multiplication maps and the inner products.
\end{definition}

\begin{theorem}
  \label{the:Fell_bundle_correspondences}
  Let \((\A,\u^\A)\) and~\((\B,\u^\B)\) be saturated Fell bundles
  over~\(\cm\).  The groupoid of Fell bundle correspondences
  from~\((\A,\u^\A)\) to~\((\B,\u^\B)\) and isomorphisms between
  such correspondences is equivalent to the groupoid of
  transformations and modifications between the \(\cm\)\nb-actions
  associated to \((\A,\u^\A)\) and~\((\B,\u^\B)\).
\end{theorem}

\begin{proof}
  Let \((\Hilm_1,\chi_g)\) be a transformation from the
  \(\cm\)\nb-action corresponding to~\(\A\) to the \(\cm\)\nb-action
  corresponding to~\(\B\).  This contains isomorphisms
  \[
  \Hilm_g \defeq \A_g\otimes_A \Hilm_1
  \xrightarrow[\cong]{\chi_{g^{-1}}}
  \Hilm_1 \otimes_B \B_g
  \]
  for \(g\in G\).  The notation~\(\Hilm_g\) leads to no serious
  confusion for \(g=1\) because \(\A_1\otimes_A \Hilm_1 = A\otimes_A
  \Hilm_1 \cong\Hilm_1\) canonically.  The spaces~\(\Hilm_g\) are
  the fibres of a correspondence \(\Gamma_0(\A)\otimes_A\Hilm_1\)
  from \(\Cont_0(G,A)\) to \(\Cont_0(G,B)\); the continuity
  of~\(\chi\) gives \(\Gamma_0(\A)\otimes_A \Hilm_1\cong
  \Hilm_1\otimes_B \Gamma_0(\B)\).  We may topologise
  \(\Hilm\defeq\bigcup_{g\in G} \Hilm_g\) so that
  \[
  \Gamma_0(\Hilm)
  = \Gamma_0(\A)\otimes_A\Hilm_1
  \cong \Hilm_1\otimes_B \Gamma_0(\B).
  \]

  The multiplications in \(\A\) and~\(\B\) induce continuous
  multiplication maps \(\A\times\Hilm\to\Hilm\)
  \(\Hilm\times\B\to\Hilm\), defining \(\Hilm_g\) as \(\A_g\otimes_A
  \Hilm_1\) and \(\Hilm_1\otimes_B \B_g\), respectively.  These
  multiplications satisfy \((a_1a_2)\xi=a_1(a_2\xi)\) for all
  \(a_1,a_2\in\A\), \(\xi\in\Hilm\) and \(\xi(b_1b_2)= (\xi b_1)b_2\)
  for all \(\xi\in\Hilm\), \(b_1,b_2\in\B\).  The bimodule property
  \((a\xi)b=a(\xi b)\) for all \(a\in\A\), \(\xi\in\Hilm\), \(b\in\B\)
  is equivalent to~\eqref{eq:twisted_conjugacy_coherence} by a routine
  computation.  The nondegeneracy of correspondences implies
  \(\Hilm_g=\A_1\cdot\Hilm_g = \Hilm_g\cdot\B_1\).

  The \(B\)\nb-valued inner product on~\(\Hilm\) induces inner
  product maps
  \[
  \langle\blank,\blank\rangle\colon
  \conj{\Hilm}_{g_1}\times\Hilm_{g_2}\to \B_{g_1^{-1}g_2},
  \]
  where~\(\conj{\Hilm}_{g_1}\) denotes the conjugate space
  of~\(\Hilm_{g_1}\) and \(\langle \xi_1\otimes b_1,\xi_2\otimes
  b_2\rangle \defeq b_1^* \langle \xi_1,\xi_2\rangle_B b_2\) for
  \(\xi_1,\xi_2\in\Hilm\), \(b_1\in\B_{g_1}\), \(b_2\in\B_{g_2}\),
  and we identify \(B=\B_1\).  The required Hilbert module
  properties of this inner product are routine to check.

  A unitary multiplier~\(u\) of~\(\A\) of degree \(g\in G\) gives an
  adjointable map \(\A_{g_2}\to\A_{gg_2}\), which induces an
  adjointable map \(\Hilm_{g_2}\to \Hilm_{gg_2}\), \(\xi\mapsto
  u\cdot\xi\).  Similarly, a unitary multiplier~\(v\) of~\(\B\) of degree
  \(g\in G\) induces maps \(\Hilm_{g_2}\to\Hilm_{g_2g}\), \(\xi\mapsto
  \xi\cdot v\), using the map \(\B_{g_2}\to\B_{g_2g}\).  In
  particular, \(\u_h^\A\) and~\(\u_h^\B\) act on~\((\Hilm_g)\) by left
  and right multiplication maps.  The
  condition~\eqref{eq:transformation_H_natural} for a transformation
  is equivalent to \(\u_h^\A\cdot\xi = \xi\cdot \u_h^\B\) for all
  \(h\in H\), \(\xi\in \Hilm\).  For \(\xi\in\Hilm_g\)
  and \(g\in G\), this implies
  \[
  \u_{\acm_g(h)}^\A\cdot\xi = \xi\cdot \u_h^\B.
  \]
  To see this, write \(\xi=a\xi_1\) with \(a\in\A_g\),
  \(\xi_1\in\Hilm_1\), and compute \(a\xi_1\cdot \u_h^B = a \u_h^\A
  \xi_1 = \u_{\acm_g(h)}^\A\cdot a\xi_1\) using
  condition~\ref{def:FellBundle2Group2} in
  Definition~\ref{def:FellBundle2Group}.  Thus a transformation
  between actions of~\(\cm\) by correspondences yields a
  correspondence between the associated Fell bundles.

  To show the converse, take a correspondence of Fell
  bundles~\(\Hilm\).  Then~\(\Hilm_g\) is a Hilbert module
  over~\(\B_1\), and the multiplication maps
  \(\Hilm_1\otimes_{\B_1} \B_g\to\Hilm_g\) must be unitary.  These
  maps are the fibres of a continuous map, so that \(\Hilm\cong
  \Hilm_1\otimes_{\B_1}\B\).  The left multiplication gives
  unitaries
  \[
  \A_g\otimes_{\A_1}\Hilm_1\cong\Hilm_g \cong \Hilm_1\otimes_{\B_1}\B_g.
  \]
  These yield the maps~\(\chi_{g^{-1}}\colon \A_g\otimes_{\A_1}\Hilm_1\congto\Hilm_1\otimes_{\B_1}\B_g\).
  Reversing the computations above, it is then routine to show that \((\Hilm_1,\chi_g)\) is a transformation of
  \(\cm\)\nb-actions.

  A modification between transformations consists of an isomorphism of
  correspondences \(\Hilm\to\Hilm'\), which then induces
  isomorphisms \(\Hilm_g=\Hilm\otimes_B \B_g\cong \Hilm'\otimes_B
  \B_g=\Hilm'_g\).  The compatibility with~\(\chi_g\) shows that these
  maps preserve also the left module structure, hence give an
  isomorphism of Fell bundle correspondences.  Conversely, an
  isomorphism of Fell bundle correspondences is determined by its
  action on the unit fibre~\(\Hilm_1\), which is a unitary~\(W\) that
  satisfies the compatibility condition~\eqref{eq:modification}.

  Thus our constructions of Fell bundle correspondences from
  \(\cm\)\nb-actions and vice versa are fully faithful functors
  between the respective categories.  These functors are inverse to
  each other up to natural isomorphisms.
\end{proof}

The invertible transformations between \(\cm\)\nb-actions are those
where~\(\Hilm\) is an imprimitivity bimodule.  In this case,
\(\Hilm\) carries an \(\A\)\nb-valued left inner product as well
such that \(_\A\langle \xi,\eta\rangle\cdot \zeta =\xi\cdot \langle
\eta, \zeta\rangle_\B\) for all \(\xi,\eta,\zeta\in\Hilm\).  This
leads to an appropriate definition of a \emph{Morita--Rieffel
  equivalence} between Fell bundles over~\(\cm\).  The functoriality
of crossed products implies that Morita--Rieffel equivalent Fell
bundles over~\(\cm\) have Morita--Rieffel equivalent crossed
products.

Now we specialise to a transformation \((\Hilm_1,\chi)\) from \(A\)
to~\(B\) where the underlying right Hilbert module of~\(\Hilm_1\)
is~\(B\) itself.  Thus the left \(A\)\nb-action on~\(\Hilm_1\) is a
nondegenerate \Star{}homomorphism \(A\to\Mult(B)\).

\begin{proposition}
  \label{pro:transformation_free}
  Transformations \((\Hilm_1,\chi)\) between two \(\cm\)\nb-algebras
  \(A\) and~\(B\) with \(\Hilm_1=B\) as a Hilbert \(B\)\nb-module
  correspond to morphisms between the Fell bundles \(\A\) and~\(\B\)
  over~\(\cm\) associated to \(A\) and~\(B\), that is, to maps
  \(f\colon \A\to\Mult(\B)\)
  \begin{itemize}
  \item that are fibrewise linear;
  \item satisfy \(f(a^*)=f(a)^*\) and
    \(f(a_1\cdot a_2)=f(a_1)\cdot f(a_2)\) for all
    \(a,a_1,a_2\in\A\);
  \item that are nondegenerate \textup(\(f(\A_1)\cdot
    \B_1=\B_1\)\textup);
  \item satisfy \(f(\u_h^\A)=\u_h^\B\) for all \(h\in H\), where nondegeneracy has been used to extend \(f\) to multipliers; and
  \item that are strictly continuous in the sense that the map
    \(\A\times\B\to\B\), \((a,b)\mapsto f(a)\cdot b\), is
    continuous.
  \end{itemize}
  The joint continuity in the last condition is equivalent to the
  continuity of the maps \(a\mapsto f(a)\cdot b\) for all
  \(b\in\B_1\).
\end{proposition}
\begin{proof}
  Let~\(\Hilm\) be the Fell bundle correspondence associated to a
  transformation \((\Hilm_1,\chi)\).  We assume that
  \(\Hilm_1=\B_1=B\) as a right Hilbert \(\B_1\)\nb-module.  This
  gives an isomorphism of Banach bundles \(\Hilm\cong\B\) that is
  right \(\B\)\nb-linear and unitary for the \(\B\)\nb-valued inner
  product because the multiplication map \((\xi,b)\mapsto \xi\cdot
  b\) induces a unitary \(\Hilm_1\otimes_{\B_1} B\cong \Hilm\).
  Conversely, a Fell bundle correspondence with \(\Hilm=\B\) as
  right Hilbert \(\B\)\nb-module has \(\Hilm_1=\B_1\) as a right
  Hilbert \(\B_1\)\nb-module.

  A Fell bundle correspondence with \(\Hilm=\B\) as a right Hilbert
  \(\B\)\nb-module is the same as a continuous multiplication
  \(\A\times\B\to\B\) with the properties in
  Definition~\ref{def:corr_Fell_bundle}.  These properties are
  equivalent to~\(f\) being a morphism \(\A\to\Mult(\B)\) of Fell
  bundles over~\(\cm\) as defined in the statement of the
  proposition.  The equivalence of joint and separate continuity
  follows because \(\norm{f(a)}\le\norm{a}\) for all \(a\in\A\).
\end{proof}

Finally, we specialise to representations of \(\cm\)\nb-actions.
These are, by definition, transformations to a
\(\Cst\)\nb-algebra~\(B\) equipped with the trivial
\(\cm\)\nb-action.  The trivial \(\cm\)\nb-action on~\(B\)
corresponds to the constant Fell bundle with fibre~\(B\) and
\(\u_h^B=1_B\) for all \(h\in H\).  A morphism of Fell bundles
from \((\A,\u)\) to such a constant Fell bundle is a map
\(\rho\colon \A\to \Mult(B)\) such that
\begin{enumerate}
\item \(\rho\) restricts to linear maps \(\A_g\to \Mult(B)\) for
  all \(g\in G\);
\item \(\rho(a^*)=\rho(a)^*\) and \(\rho(a_1\cdot
  a_2)=\rho(a_1)\cdot \rho(a_2)\) for all \(a,a_1,a_2\in \A\);
\item for each \(b\in B\), the map \(a\mapsto \rho(a)b\) is
  continuous from~\(\A\) to~\(B\) (continuity);
\item \(\rho(\A_1)\cdot B=B\) (nondegeneracy);
\item \(\rho(\u_h^\A)=1\) for all \(h\in H\);
\end{enumerate}
the last condition involves the canonical extension of a
representation of a Fell bundle to multipliers.  Except for the last
condition, this is the standard definition of a representation of a
Fell bundle (see \cite{Doran-Fell:Representations_2}).

We are going to describe the crossed product of a \(\cm\)\nb-action
by correspondences as a quotient of the cross-sectional
\(\Cst\)\nb-algebra of the associated Fell bundle, generalising the
description in~\cite{Buss-Meyer-Zhu:Non-Hausdorff_symmetries} for
strict \(\cm\)\nb-actions.  The cross-sectional \cstar{}algebra
\(\Cst(\A)\) of a Fell bundle~\(\A\) over the locally compact
group~\(G\) is constructed in
\cite{Doran-Fell:Representations_2}*{Chapter~VIII} as the
\(\Cst\)\nb-completion of the space~\(\contc(\A)\) of compactly
supported sections of~\(\A\) with a suitable convolution and
involution.  It comes with a canonical map \(\rho^\un\colon
\A\to\Mult(\Cst(\A))\), which is a representation of~\(\A\) viewed
as a \(G\)\nb-\(\Cst\)-algebra (via
Theorem~\ref{theo:ActionCorrespondences=FellBundles}).

\begin{definition}
  \label{def:CrossSectionalAlgebraForFellBundlesOverCrossedModules}
  For a Fell bundle~\((\A,\u)\) over~\(\cm\), let~\(\I_\u\) be the
  closed, two-sided ideal of~\(\Cst(\A)\) generated by the
  multipliers \(\{\rho^\un(\u_h)-1: h\in H\}\), that is,
  \[
  \I_\u\defeq \cspn\{x(\rho^\un(\u_h)-1)y:x,y\in \Cst(\A), h\in H\}.
  \]
  We call \(\Cst(\A,\u)\defeq \Cst(\A)/\I_\u\) the
  \emph{cross-sectional \cstar{}algebra} of~\((\A,\u)\).
\end{definition}

Since \(\rho(\u_h)=1\) is the only extra condition needed for a
representation of~\(\A\) to be a representation of \((\A,\u)\),
\(\Cst(\A,\u)\) is the largest quotient of~\(\Cst(\A)\) on
which~\(\rho^\un\) gives a representation of~\((\A,\u)\) as a
\(\cm\)\nb-\(\Cst\)-algebra.

\begin{proposition}
  \label{pro:crossed_product}
  The \cstar{}algebra \(\Cst(\A,\u)\defeq \Cst(\A)/\I_\u\) with the
  canonical representation of~\(\A\) is a crossed product
  \((\A,\u)\rtimes \cm\).  That is, representations of
  \(\Cst(\A,\u)\) correspond bijectively to representations of
  \((\A,\u)\).
\end{proposition}

\begin{proof}
  Morphisms \(\Cst(\A)\to B\) are in bijection with
  representations of the Fell bundle~\(\A\) on~\(B\), mapping a
  representation~\(\rho\) of~\(\A\) to its integrated
  form~\({\smallint}\rho\) (see \cite{Doran-Fell:Representations_2}*{VIII.11}).
  Representations of \(\Cst(\A,\u)=\Cst(\A)/\I_\u\) correspond
  bijectively to representations of~\(\Cst(\A)\) that vanish
  on~\(\I_\u\).  Moreover, for a representation~\(\pi\) of~\(\A\),
  \begin{align*}
    {\smallint}\rho(\I_\u)=0
    &\iff
    {\smallint}\rho(\rho^\un(\u_h)\xi-\xi)=0
    \text{ for all }h\in H,\ \xi\in \contc(\A)\\
    &\iff
    \rho(\u_h\cdot a)=\rho(a)
    \text{ for all }a\in \A.
  \end{align*}
  Thus the bijection from representations of~\(\A\) on~\(B\) to
  morphisms \(\Cst(\A)\to B\) restricts to a bijection from
  representations of \((\A,\u)\) to morphisms \(\Cst(\A,\u)\to B\).
\end{proof}

\section{Equivalence of crossed modules}
\label{sec:equivalence_crossed_modules}

One should expect equivalent crossed modules to have equivalent
\(2\)\nb-categories of actions on \(\Cst\)\nb-algebras by correspondences.
What does this mean?  A functor \(\Corr(\cm)\to\Corr(\cm')\) for two
crossed modules \(\cm\) and~\(\cm'\) consists of
\begin{itemize}
\item a map~\(F\) on objects, which maps each \(\cm\)\nb-action
  \(x\) to a \(\cm'\)\nb-action \(F(x)\);
\item for any two \(\cm\)\nb-actions \(x_1\) and~\(x_2\), a functor
  \[
  F\colon \Corr_\cm(x_1,x_2) \to
  \Corr_{\cm'}(F(x_1),F(x_2));
  \]
\item for any three \(\cm\)\nb-actions \(x_1\), \(x_2\), \(x_3\),
  natural isomorphisms \(F(f)\circ F(g)\cong F(f\circ g)\) in
  \(\Corr_{\cm'}(F(x_1),F(x_3))\) for \(f\in\Corr_\cm(x_2,x_3)\),
  \(g\in\Corr_\cm(x_1,x_2)\);
\item natural isomorphisms \(F(1_x)\cong 1_{F(x)}\);
\end{itemize}
the natural transformations in the last two conditions must satisfy
suitable coherence axioms (see~\cite{Leinster:Basic_Bicategories}).
Such a functor is an equivalence if the map~\(F\) on objects is
essentially surjective and the functors on the arrow groupoids are
equivalences.

An equivalence of \(2\)\nb-categories has a quasi-inverse that is again a
functor, such that the compositions in either order are
equivalent to the identity functor in a suitable sense.

We will only consider functors constructed from homomorphisms of
crossed modules.  Let \(\cm_i=(G_i,H_i,\tcm_i,\acm_i)\) for \(i=1,2\)
be crossed modules.  A \emph{homomorphism of crossed modules}
\(\cm_1\to\cm_2\) is a pair of continuous group homomorphisms
\(\varphi\colon G_1\to G_2\), \(\psi\colon H_1\to H_2\) such that
\(\tcm_2\circ\psi=\varphi\circ\tcm_1\) and \(\acm_{2,\varphi(g)}(\psi(h)) =
\psi(\acm_{1,g}(h))\).  Such a homomorphism induces a functor
\begin{equation}\label{eq:Functor_Induced_Morphism}
(\varphi,\psi)^*\colon \Corr(\cm_2)\to\Corr(\cm_1),
\end{equation}
by sending a \(\cm_2\)\nb-action \((A,\alpha,\omega,\eta)\) to
\((A,\varphi^*\alpha,(\varphi\times\varphi)^*\omega,\psi^*\eta)\),
a transformation \((\E,\chi)\) to \((\E,\varphi^*\chi)\), and a
modification~\(W\) again to~\(W\); the natural transformations in
the definition of a functor above are trivial and therefore
coherent.

When we translate to Fell bundles using
Theorem~\ref{theo:ActionCorrespondences=FellBundles}, the
functor~\eqref{eq:Functor_Induced_Morphism} sends a Fell
bundle~\((\A,\u)\) over~\(\cm_1\) to the Fell bundle
\((\varphi^*\A,\psi^*\u)\) over~\(\cm_2\), where~\(\varphi^*\A\)
denotes the pull-back of~\(\A\) along~\(\varphi\) and
\(\psi^*\u=\u\circ\psi\).

\begin{definition}
  \label{def:arrow_groupoid}
  The \emph{arrow groupoid} of \(\cm=(G,H,\tcm,\acm)\) is the
  transformation groupoid \(G\rtimes H\) for the right action of the
  topological group~\(H\) on the topological space~\(G\) defined by
  \(g\cdot h\defeq g\tcm(h)\).

  A homomorphism \((\varphi,\psi)\) of crossed modules induces a
  functor between the arrow groupoids.  We call \((\varphi,\psi)\)
  an \emph{equivalence} if the induced functor on arrow groupoids
  gives an equivalence of topological groupoids.
\end{definition}

To make this precise, we turn a functor between the arrow groupoids
into a Hilsum--Skandalis morphism, that is, a topological space
endowed with commuting actions of \(G_1\rtimes H_1\) and
\(G_2\rtimes H_2\) (see~\cite{Hilsum-Skandalis:Morphismes}).  Let
\(X\defeq G_1\times H_2\) and define anchor maps \(\pi_i\colon X\to
G_i\) for \(i=1,2\) by \(\pi_1(g_1,h_2)\defeq g_1\) and
\(\pi_2(g_1,h_2)\defeq \varphi(g_1)\tcm_2(h_2)\).  Define a left
\(H_1\)\nb-action and a right \(H_2\)\nb-action on~\(X\) by
\begin{alignat*}{2}
  h\cdot (g_1,h_2) &\defeq (g_1\tcm_1(h)^{-1},\psi(h)h_2)
  &\qquad &\text{for }h\in H_1,\\
  (g_1,h_2)\cdot h &\defeq (g_1,h_2h)
  &\qquad &\text{for }h\in H_2.
\end{alignat*}
These definitions turn~\(X\) into a \(G_1\rtimes H_1\)-\(G_2\rtimes
H_2\)-bispace.  This bispace is always a Hilsum--Skandalis morphism,
that is, the action of \(G_2\rtimes H_2\) is free and proper
and~\(\pi_1\) induces a homeomorphism from the \(G_2\rtimes
H_2\)-orbit space to~\(G_1\).  The homomorphism \((\varphi,\psi)\) is
an equivalence if the associated bispace is a Morita equivalence as
in~\cite{Muhly-Renault-Williams:Equivalence}.  In our case, this
happens if and only if the action of \(G_1\rtimes H_1\) on~\(X\) is
free and proper with orbit space projection~\(\pi_2\).  Equivalently,
\(X\) is a Morita equivalence between \(G_1\rtimes H_1\) and
\(G_2\rtimes H_2\) as in~\cite{Muhly-Renault-Williams:Equivalence}.
It is enough to check that~\(H_1\) acts freely and properly on~\(X\)
with \(X/H_1\cong G_2\) via~\(\pi_2\).

\begin{lemma}
  \label{lem:cm_equivalence_criterion}
  The homomorphism \((\varphi,\psi)\) is an equivalence if and only if
  \begin{enumerate}
  \item \label{en:cm_equivalence_1} the map
    \[
    (\tcm_1,\psi)\colon
    H_1\to \{(g_1,h_2)\in G_1\times H_2 : \varphi(g_1)=\tcm_2(h_2) \}
    = G_1\times_{G_2} H_2
    \]
    is a homeomorphism, where the codomain carries the subspace
    topology;
  \item \label{en:cm_equivalence_2} the map \(\pi_2\colon G_1\times
    H_2\to G_2\), \((g_1,h_2)\mapsto \varphi(g_1)\cdot
    \tcm_2(h_2)\), is an open surjection.
  \end{enumerate}
\end{lemma}

The first condition says that~\(H_1\) is the pull-back of~\(H_2\)
along~\(\varphi\), the second condition is a transversality
condition for this pull-back.

\begin{proof}
  The freeness of the \(H_1\)\nb-action on~\(X\) means that no
  \(h\neq1\) in~\(H_1\) has \(\psi(h)=1\) and \(\tcm_1(h)=1\), that
  is, the map in~\ref{en:cm_equivalence_1} is injective.

  Let \((g_1,h_2)\in G_1 \times H_2\) and \((g_1',h_2')\in G_1
  \times H_2\).  They have the same \(\pi_2\)\nb-image if and only
  if \(\varphi(g_1)\tcm_2(h_2) = \varphi(g_1')\tcm_2(h_2')\), if and
  only if \(((g_1')^{-1}g_1,h_2'h_2^{-1})\in G_1\times H_2\)
  satisfies \(\varphi((g_1')^{-1}g_1) = \tcm_2(h_2'h_2^{-1})\);
  hence the surjectivity of the map in~\ref{en:cm_equivalence_1}
  means that the map \(\pi_2'\colon X/H_1\to G_2\) induced
  by~\(\pi_2\) is injective.  The map~\(\pi_2'\) is automatically
  continuous, and it is open if and only if~\(\pi_2\) is because the
  projection \(X\to X/H_1\) is open and continuous.  Hence
  condition~\ref{en:cm_equivalence_2} means that~\(\pi_2'\) is
  surjective and open.  Being injective and continuous as well, it
  is a homeomorphism.

  Finally, the properness of the \(H_1\)\nb-action is equivalent to
  the following: for all compact subsets \(K\subseteq G_1\),
  \(L\subseteq H_2\), the set of \(h\in H_1\) with \(\tcm_1(h)\in
  K^{-1}\cdot K\) and \(\psi(h)\in L\cdot L^{-1}\) is compact.
  This is equivalent to the properness of the map
  in~\ref{en:cm_equivalence_1}.  We already know that this map is a
  continuous bijection.  Such a map is proper if and only if it is a
  homeomorphism.
\end{proof}

For discrete crossed modules,
\cite{Noohi:two-groupoids}*{Proposition 6.3} says that the
equivalences in the above sense are the acyclic cofibrations in a
suitable model structure.

The arrow groupoid does not yet encode the multiplication in~\(G\)
and the conjugation action~\(\acm\).  These are encoded in a
continuous functor
\begin{align}
  \label{eq:mult_arrow_groupoid}
  M\colon (G\rtimes H) \times (G\rtimes H)\to (G\rtimes H),
  \qquad
  (g_1,g_2)&\mapsto g_1\cdot g_2,\\
  (g_1,h_1),(g_2,h_2)&\mapsto
  (g_1\cdot g_2,\acm_{g_2}^{-1}(h_1)\cdot h_2); \notag
\end{align}
here \((g,h)\) denotes the \(2\)\nb-arrow \(h\colon g\Rightarrow
g\tcm(h)\).  The existence and associativity of this functor is
equivalent to the axioms of a crossed module.  The orbit space of
the arrow groupoid is \(\pi_1(\cm)\defeq \coker\tcm\); the group
structure on the orbit space is induced by~\(M\); the isotropy group
of any \(g\in G\) is isomorphic to \(\pi_2(\cm)\defeq \ker\tcm\).
This is an Abelian group, and the action~\(\acm\) induces a
\(\pi_1(\cm)\)-module structure on it; this module structure may
also be expressed through~\(M\).  Since we may express them through
canonical extra structure on the arrow groupoid, the group
\(\pi_1(\cm)\) and the \(\pi_1(\cm)\)-module \(\pi_2(\cm)\) are
invariant under equivalences; for crossed modules of locally compact
groups, this invariance includes the induced topologies on them.

Discrete crossed modules are classified up to equivalence by the
group \(\pi_1(\cm)\), the \(\pi_1(\cm)\)-module \(\pi_2(\cm)\), and
a cohomology class in \(\textup{H}^3(\pi_1(\cm),\pi_2(\cm))\) (see
\cite{MacLane-Whitehead:3-type}).  We shall not attempt such a
complete classification of crossed modules of locally compact groups
here.  It is useful, however, to know that \(\pi_1\) and~\(\pi_2\)
are invariant under equivalence.

\begin{theorem}
  \label{the:equivalence_for_actions}
  The functor \(\Corr(\cm_2)\to\Corr(\cm)\) induced by an
  equivalence of crossed modules \((\varphi,\psi)\colon
  \cm\to\cm_2\) is an equivalence of \(2\)\nb-categories.

  This equivalence intertwines the crossed product functors on both
  categories.
\end{theorem}

\begin{proof}
  We must show that any \(\cm\)\nb-action
  \((A,\alpha,\omega,\eta)\) is isomorphic to
  \[
  (\varphi,\psi)^*(A_2,\alpha_2,\omega_2,\eta_2)
  \]
  for an action of~\(\cm_2\) and that any transformation
  \((\E,\chi)\) between \(\cm\)\nb-actions is isomorphic to
  \((\varphi,\psi)^*(\E_2,\chi_2)\) for a transformation of
  \(\cm_2\)\nb-actions \((\E_2,\chi_2)\) that is unique up to
  isomorphism.  Since \((\varphi,\psi)^*\) does not change the
  underlying \(\Cst\)\nb-algebras, we may put \(A_2\defeq A\).

  Define isomorphisms of correspondences \(\eta_{g,h}\colon
  \alpha_g\to\alpha_{g\tcm(h)}\) for \(g\in G\), \(h\in H\) by
  \[
  \eta_{g,h}\colon
  \alpha_g \cong A\otimes_A \alpha_g
  \xrightarrow{\eta_h\otimes_A\Id_{\alpha_g}}
  \alpha_{\tcm(h)}\otimes_A\alpha_g
  \xrightarrow{\omega(g,\tcm(h))}
  \alpha_{g\tcm(h)}.
  \]
  The space \(G\times H\) is the space of \(2\)\nb-arrows
  in~\(\cm\), and an action in the sense
  of~\cite{Buss-Meyer-Zhu:Higher_twisted}*{Definition 4.1}
  provides~\(\eta_{g,h}\) as above.  We removed redundancy in
  Definition~\ref{def:act_corr} and kept only \(\eta_{1,h}=\eta_h\).
  Now we need the whole family~\(\eta_{g,h}\).  It depends
  continuously on \((g,h)\in G\times H\).  The naturality of the
  isomorphisms~\(\omega_{g_1,g_2}\) with respect to \(2\)\nb-arrows
  says that the diagrams
  \begin{equation}
    \label{eq:large_coherence_eta_omega}
    \begin{tikzpicture}[yscale=1.2,xscale=5,
      baseline=(current bounding box.west)]
      \node (m-0-1) at (0,1) {\(\alpha_{g_2}\otimes_A \alpha_{g_1}\)};
      \node (m-0-2) at (1,1) {\(\alpha_{g_1g_2}\)};
      \node (m-1-1) at (0,0)
      {\(\alpha_{g_2\tcm(h_2)}\otimes_A \alpha_{g_1\tcm(h_1)}\)};
      \node (m-1-2) at (1,0) {\(\alpha_{g_1\tcm(h_1)\cdot g_2\tcm(h_2)}\)};
      \draw[cdar] (m-0-1) -- node {\(\omega_{g_1,g_2}\)} (m-0-2);
      \draw[cdar] (m-1-1) -- node[swap]
      {\(\omega_{g_1\tcm(h_1),g_2\tcm(h_2)}\)} (m-1-2);
      \draw[cdar] (m-0-1) -- node[swap]
      {\(\eta_{g_2,h_2}\otimes_A\eta_{g_1,h_1}\)} (m-1-1);
      \draw[cdar] (m-0-2) -- node
      {\(\eta_{g_1g_2,\acm_{g_2}^{-1}(h_1)h_2}\)} (m-1-2);
    \end{tikzpicture}
  \end{equation}
  commute for all \(g_1,g_2\in G\), \(h_1,h_2\in H\).

  The isomorphisms \(\eta_{g,h}\colon \alpha_g \to
  \alpha_{g\tcm(h)}\) for \(g\in G\), \(h\in H\) turn~\(\alpha\)
  into a \(G\rtimes H\)-equivariant Hilbert bimodule over
  \(\Cont_0(G,A)\).  Since \(G\rtimes H\) is Morita equivalent to
  \(G_2\rtimes H_2\) via the bispace~\(X\), we may transport this to
  an \(G_2\rtimes H_2\)-equivariant Hilbert bimodule~\(\alpha_2\)
  over \(\Cont_0(G_2,A)\).  We describe~\(\alpha_2\) more
  explicitly.

  The pull-back \(\pi_1^*\alpha \cong \alpha\otimes\Cont_0(H_2)\)
  along \(\pi_1\colon X\to G\) is an \(H\times H_2\)\nb-equivariant
  Hilbert module over \(\Cont_0(X,A)\).  Let \(\alpha_2\subseteq
  \Mult(\pi_1^*\alpha)\) be the space of all bounded continuous
  sections of~\(\pi_1^*\alpha\) that are \(H\)\nb-invariant and
  whose norm function is in \(\Cont_0(X/H)\).  Since \(X/H\cong
  G_2\) via~\(\pi_2\) and the actions of \(H_2\) and~\(H\) commute,
  we may view this as an \(H_2\)\nb-equivariant Hilbert
  \(\Cont_0(G_2,A)\)-module.

  The left \(A\)\nb-action survives these constructions because the
  action of~\(H\) commutes with it.  The fibre of~\(\alpha_2\) at
  \(g_2\in G_2\) is isomorphic to~\(\alpha_g\) for any \(g\in G\)
  for which there is \(h_2\in H_2\) with \(\pi_2(g,h_2)=g_2\), that
  is, \(g_2 = \varphi(g)\tcm_2(h_2)\).  These isomorphisms on the
  fibres are continuous and give a canonical \(H\)\nb-equivariant
  isomorphism \(\varphi^*\alpha_2\cong\alpha\).

  The pull-backs \(\pi_1^*\alpha\), \(\pi_2^*\alpha\)
  and~\(\mu^*\alpha\) in the definition of~\(\omega\) are \(H\times
  H\)\nb-equivariant Hilbert modules over \(\Cont_0(G\times G,A)\).
  That is, they are representations of the groupoid \((G\rtimes H)^2
  = (G\times G)\rtimes (H\times H)\).  This groupoid is equivalent
  to \((G_2\rtimes H_2)^2\) via the bispace~\(X^2\).  Now pull back
  \(\pi_1^*\alpha\), \(\pi_2^*\alpha\), \(\mu^*\alpha\) and the
  isomorphism \(\omega\colon \pi_2^*\alpha\otimes_{\Cont_0(G\times
    G,A)} \pi_1^*\alpha \to \mu^*\alpha\) to~\(X^2\) and push all
  this down to the category of \(H_2^2\)\nb-equivariant Hilbert
  \(\Cont_0(G_2^2,A)\)-modules by taking \(H^2\)\nb-invariants.  The
  resulting Hilbert \(\Cont_0(G_2^2,A)\)-modules are canonically
  isomorphic to \(\pi_1^*\alpha_2\), \(\pi_2^*\alpha_2\),
  and~\(\mu^*\alpha_2\), respectively.  Hence~\(\omega\) induces a
  \(\Cont_0(G_2\times G_2)\)-linear \(H_2\times H_2\)\nb-equivariant
  unitary operator
  \[
  \omega_2\colon
  \pi_2^*\alpha_2\otimes_{\Cont_0(G_2\times G_2,A)} \pi_1^*\alpha_2
  \to \mu^*\alpha_2.
  \]
  The \(H_2\times H_2\)\nb-equivariance shows that~\(\omega_2\)
  satisfies the analogue of~\eqref{eq:large_coherence_eta_omega},
  which gives conditions \ref{def:act_corr2} and~\ref{def:act_corr3}
  in Definition~\ref{def:act_corr}.  To prove that~\(\omega_2\)
  inherits condition~\ref{def:act_corr1} in
  Definition~\ref{def:act_corr}, we use that for each
  \(g_1,g_2,g_3\in G_2\) there are \(g_1',g_2',g_3'\in G\) such that
  \(\omega_{2,g_1,g_2}=\omega_{g_1',g_2'}\),
  \(\omega_{2,g_1g_2,g_3}=\omega_{g_1'g_2',g_3'}\),
  \(\omega_{2,g_2,g_3}=\omega_{g_2',g_3'}\),
  \(\omega_{2,g_1,g_2g_3}=\omega_{g_1',g_2'g_3'}\) with suitable
  natural identifications of the fibres
  \(\alpha_{2,g_i}\cong\alpha_{g_i'}\) and
  \(\alpha_{2,g_ig_j}\cong\alpha_{g_i'g_j'}\).  We also find a
  canonical isomorphism \((\varphi\times\varphi)^*\omega_2\cong
  \omega\).  Hence \((\varphi,\psi)^*(A,\alpha_2,\omega_2,\eta_2) =
  (A,\alpha,\omega,\eta)\).

  In the same way, we may also lift a transformation
  \((\Hilm,\chi)\) of \(\cm\)\nb-actions to one between the
  corresponding \(\cm_2\)\nb-actions.  We may put \(\Hilm_2=\Hilm\)
  because our lifting does not change the underlying
  \(\Cst\)\nb-algebras.  The assumptions in
  Definition~\ref{def:transformation} imply that the
  isomorphism~\(\chi\) is \(H\)\nb-equivariant.  Hence we may pull
  it back to an \(H\times H_2\)\nb-equivariant isomorphism
  over~\(X\) and then push down to~\(G_2\) to get an
  \(H_2\)\nb-equivariant isomorphism \(\chi_2\colon \Hilm\otimes_B
  \beta_2\to \alpha_2\otimes_A \Hilm\) over~\(G_2\).  The same arguments
  as above show that \((\Hilm,\chi_2)\) is a transformation that is
  a \((\varphi,\psi)^*\)-preimage of the transformation
  \((\Hilm,\chi)\), and the only one with this property up to
  isomorphism of transformations.

  It is clear that \((\varphi,\psi)^*\) maps a trivial action
  of~\(\cm_2\) to a trivial action of~\(\cm\).  Furthermore, on the
  level of transformations, \(\Hilm\) is a morphism (its underlying
  Hilbert module is~\(B\) itself) if and only if
  \((\varphi,\psi)^*\Hilm\) is a morphism.  Therefore,
  \((\varphi,\psi)^*\) gives a bijection between
  \(\cm_2\)\nb-representations of \((A,\alpha_2,\omega_2,\eta_2)\)
  on~\(B\) and \(\cm\)\nb-representations of
  \((A,\alpha,\omega,\eta)\) on~\(B\).  By the universal property,
  there is a unique isomorphism \(A\rtimes_{\alpha,\omega,\eta} \cm
  \cong A\rtimes_{\alpha_2,\omega_2,\eta_2} \cm_2\) that maps the
  universal \(\cm\)\nb-representation to the universal
  \(\cm_2\)\nb-representation.
\end{proof}

\begin{example}
  \label{exa:cm-equivalence_enlarge}
  Let \(\cm=(G,H,\tcm,\acm)\) be a crossed module and let
  \(G_1\subseteq G\) be a closed subgroup such that the map
  \(G_1\times H\to G\), \((g,h)\mapsto g\cdot\tcm(h)\), is open and
  surjective.  Let \(H_1\defeq \tcm^{-1}(G_1)\) and let \(\tcm_1\colon
  H_1\to G_1\) and \(\acm_1\colon G_1\to\Aut(H_1)\) be the
  restrictions of \(\tcm\) and~\(\acm\).  Then the embedding
  \((\varphi,\psi)\) of \(\cm_1=(G_1,H_1,\tcm_1,\acm_1)\) into~\(\cm\)
  is an equivalence by Lemma~\ref{lem:cm_equivalence_criterion}.  The
  second condition in Lemma~\ref{lem:cm_equivalence_criterion} is our
  assumption.  The first condition is that the group homomorphism
  \[
  H_1 \to \{(g_1,h)\in G_1\times H : \varphi(g_1)=\tcm(h) \},
  \qquad
  h_1\mapsto (\psi(h_1),\tcm_1(h_1)),
  \]
  is a homeomorphism.  Indeed, the map has the restriction of the
  first coordinate projection as a continuous inverse.

  Theorem~\ref{the:equivalence_for_actions} says that a Fell bundle
  over~\(\cm_1\) extends to one over~\(\cm\) in a natural and
  essentially unique way.
\end{example}

\begin{example}
  \label{exa:cm-equivalence_quotient}
  Let \(N\subseteq H\) be a closed \(\acm(G)\)-invariant subgroup such
  that~\(\tcm\) restricts to a homeomorphism from~\(N\) onto a closed
  subgroup of~\(G\); then \(\tcm(N)\) is normal in~\(G\).  Let
  \(G_2\defeq G/\tcm(N)\), \(H_2\defeq H/N\), and let \(\tcm_2\colon
  H_2\to G_2\) and \(\acm_2\colon G_2\to\Aut(H_2)\) be the induced
  maps.  Then the projection map \((\varphi,\psi)\) from~\(\cm\) to
  \(\cm_2=(G_2,H_2,\tcm_2,\acm_2)\) is an equivalence by
  Lemma~\ref{lem:cm_equivalence_criterion}.  The second condition in
  Lemma~\ref{lem:cm_equivalence_criterion} follows because already the
  projection \(G\to G_2\) is open and surjective.  The first condition
  requires the map
  \[
  H\to \{(g,h_2)\in G\times H_2 : \varphi(g)=\tcm_2(h_2) \},
  \qquad
  h\mapsto (\psi(h),\tcm(h)),
  \]
  to be a homeomorphism.  Injectivity and surjectivity of this map are
  clear, and its properness is not hard to check.  This implies that
  the map is a homeomorphism.

  Theorem~\ref{the:equivalence_for_actions} says that any Fell
  bundle over~\(\cm\) is the pull-back of a Fell bundle
  over~\(\cm_2\), which is unique up to isomorphism and depends
  naturally on the original Fell bundle over~\(\cm\).  The reason to
  expect this is that the unitary multipliers~\(\u_h\) for \(h\in
  N\) trivialise our Fell bundle over \(N\)\nb-cosets.
\end{example}

\begin{example}
  \label{exa:grouplike_classical}
  Even more specially, assume that~\(\cm\) is group-like, that is,
  \(\tcm\)~is a homeomorphism onto a closed (normal) subgroup of~\(G\).
  Theorem~\ref{the:equivalence_for_actions} says that
  \(\cm\)\nb-actions are equivalent to actions of the locally compact
  group~\(G/\tcm(H)\) viewed as a crossed module.  Actions of the
  latter are equivalent to Fell bundles over the locally compact
  group~\(G/\tcm(H)\) in the usual sense.  Actions of~\(\cm\) are a
  Fell bundle analogue of Green twisted actions.  In particular, any
  Green twisted action of \((G,\tcm(H))\) gives rise to a Fell bundle
  over~\(G/\tcm(H)\).

  For discrete groups, this equivalence is already contained
  in~\cite{Echterhoff-Quigg:InducedCoactions}.
\end{example}

We have defined when a crossed module homomorphism is an
equivalence.  Since its inverse is usually not described by a
crossed module homomorphism, we are led to the following equivalence
relation for crossed modules:

\begin{definition}
  \label{def:cm_equivalent}
  Two crossed modules of locally compact groups \(\cm\) and~\(\cm'\)
  are \emph{equivalent} if they are connected by a chain of crossed
  module homomorphisms
  \[
  \cm=\cm_0 \leftarrow \cm_1 \rightarrow \cm_2
  \leftarrow \cm_3 \rightarrow \dotsb \leftarrow \cm'
  \]
  where each arrow is a crossed module homomorphism that is an equivalence.
\end{definition}

Since equivalence of \(2\)\nb-categories formulated in terms of functors
is a symmetric relation (quasi-inverses exist and are again
functors), equivalent crossed modules have equivalent action
\(2\)\nb-categories \(\Corr(\cm)\) by
Theorem~\ref{the:equivalence_for_actions}.  The topological group
\(\pi_1(\cm)\) and the topological \(\pi_1(\cm)\)-module
\(\pi_2(\cm)\) are invariant under equivalence because, as already mentioned,
equivalences implemented by homomorphisms preserve \(\pi_1\) and~\(\pi_2\).

\subsection{Simplification of thin crossed modules}
\label{sec:simplification_thin}

We call a crossed module \(\cm=(G,H,\tcm,\acm)\) \emph{thin}
if~\(\tcm\) is injective and has dense range.  In the discrete case,
this implies that~\(\cm\) is equivalent to the trivial crossed
module.  An interesting example is the crossed
module associated to a dense embedding \(\Z\to\Torus\), which acts
on the corresponding noncommutative torus
(see~\cite{Buss-Meyer-Zhu:Non-Hausdorff_symmetries}).  Being thin is
an invariant of equivalence of crossed modules.  In a thin crossed
module, the action~\(\acm\) is dictated by
\(\tcm(\acm_g(h))=g\tcm(h)g^{-1}\) because~\(\tcm\) is injective.

We want to simplify thin crossed modules up to equivalence.  The
best result is available if both \(G\) and~\(H\) are Lie groups (we
allow an arbitrary number of connected components).  In that case,
the following theorem gives a complete classification up to
equivalence.

\begin{theorem}
  \label{the:thin_Lie_equivalent_Abelian}
  Any thin crossed module of Lie groups is equivalent to one where
  \(G=\R^n\) for some \(n\in\N\) and~\(H\) is a dense subgroup
  of~\(G\) with the discrete topology.  And it is equivalent to one
  where \(G=\Torus^n\) for some \(n\in\N\) and~\(H\) is a dense
  subgroup of~\(G\) with the discrete topology.

  Two thin crossed modules \(\cm_i=(G_i,H_i,\tcm_i,\acm_i)\) with
  \(G_i=\R^{n_i}\) for \(i=1,2\) and discrete~\(H_i\) are equivalent
  if and only if there is an invertible linear map
  \(\R^{n_1}\to\R^{n_2}\) mapping \(\tcm_1(H_1)\)
  onto~\(\tcm_2(H_2)\).
\end{theorem}

\begin{proof}
  Let \(\cm=(G,H,\tcm,\acm)\) be a thin crossed module of Lie
  groups.  Let \(G_1\subseteq G\) be the connected component of the
  identity.  Let \(H_1= \tcm^{-1}(G_1)\subseteq H\).  Since
  \(\tcm(H)\) is dense in~\(G\), it meets every connected component.
  Hence we are in the situation of
  Example~\ref{exa:cm-equivalence_enlarge}, and~\(\cm\) is
  equivalent to \(\cm_1=(G_1,H_1,\tcm_1,\acm_1)\), where \(\tcm_1\)
  and~\(\acm_1\) are the restrictions of \(\tcm\) and~\(\acm\).  The
  Lie group~\(G_1\) is connected.

  Let \(\varphi\colon G_2\to G_1\) be the universal covering
  of~\(G_1\) and let
  \[
  H_2 \defeq \{(h_1,g_2) \in H_1\times G_2 : \tcm_1(h_1) = \varphi(g_2)\}
  \]
  be its pull-back to a covering of~\(H_1\).  Let \(\psi\colon H_2\to
  H_1\) and \(\tcm_2\colon H_2\to G_2\) be the coordinate projections.
  Then~\(\psi\) is a covering map; its kernel is a discrete central
  subgroup \(N\subseteq H_2\).  The homomorphism~\(\tcm_2\) maps~\(N\)
  homeomorphically onto the kernel of the covering map~\(\varphi\).
  Hence we are in the situation of
  Example~\ref{exa:cm-equivalence_quotient}, and~\(\cm_1\) is
  equivalent to \(\cm_2= (G_2,H_2,\tcm_2,\acm_2)\), where~\(\acm_2\)
  lifts~\(\acm_1\).  (The composite \(\cm_2\to\cm_1\to\cm\) is also
  an equivalence, so we could have gone to~\(\cm_2\) in only one
  step.)

  The Lie group~\(G_2\) is simply connected.  Let~\(N_2\) be the
  connected component of the identity in~\(H_2\).  We claim
  that~\(\tcm_2\) maps~\(N_2\) homeomorphically onto a closed subgroup
  of~\(G_2\).  We are going to prove this later, so let us assume this
  for a moment.  Then we are again in the situation
  of Example~\ref{exa:cm-equivalence_quotient}.  Letting \(G_3\defeq
  G_2/\tcm_2(N_2)\), \(H_3\defeq H_2/N_2\), and taking the induced maps
  \(\tcm_3\) and~\(\acm_3\), we get a crossed module equivalence
  \(\cm_2\to\cm_3=(G_3,H_3,\tcm_3,\acm_3)\).

  The Lie group~\(H_3\) is discrete, and~\(G_3\) is still a simply
  connected Lie group because it is a quotient of a simply connected
  group by a closed, connected, normal subgroup.  The conjugation
  action \(\acm_3\colon G_3\to\Aut(H_3)\) is a continuous action of a
  connected group on a discrete space.  It must be trivial.  Hence
  \(g\tcm_3(h)g^{-1}=\tcm_3(h)\) for all \(h\in H_3\).  Since we
  started with a thin crossed module, \(\tcm_3(H_3)\) is dense
  in~\(G_3\), so the above equality extends to all of~\(G_3\), proving
  that~\(G_3\) is Abelian.  Any Abelian simply connected Lie group is
  of the form~\(\R^n\) for some \(n\in\N\).  Thus~\(\cm_3\) is a
  crossed module equivalent to~\(\cm\) with \(G_3=\R^n\) and
  discrete~\(H_3\).

  Now we show that~\(\tcm_2\) is a homeomorphism from~\(N_2\) onto a
  closed subgroup of~\(G_2\) if~\(G_2\) is simply connected.  Let
  \(\Lah\) and~\(\Lag\) be the Lie algebras
  of~\(H_2\) and~\(G_2\), respectively.  The dense
  embedding~\(\tcm_2\) induces an injective map
  \(\Lah\to\Lag\), whose image is a Lie ideal.
  Hence there is a Lie algebra \(\Lag/\Lah\) and a
  Lie algebra homomorphism
  \(\Lag\to\Lag/\Lah\).  By Lie's theorems,
  there is a simply connected Lie group~\(K\) with Lie algebra
  \(\Lag/\Lah\) and a Lie group homomorphism
  \(G_2\to K\) that induces
  \(\Lag\to\Lag/\Lah\) on the Lie algebras.
  The kernel of this homomorphism is a connected, closed, normal
  subgroup, and its Lie algebra is~\(\Lah\).  This closed
  normal subgroup is~\(\tcm_2(N_2)\) because the exponential map is
  a local homeomorphism from the Lie algebra to the Lie group near
  the identity element.  Moreover, the map~\(\tcm_2\) is indeed a
  homeomorphism from~\(N_2\) onto this closed normal subgroup.

  To get another equivalent crossed module~\(\cm_4\) with
  \(G_4\cong\Torus^n\), we choose elements \(h_1,\dotsc,h_n\in H_3\)
  whose images form a basis in~\(\R^n\), using the density of
  \(\tcm(H_3)\) in \(G_3=\R^n\).  These elements generate a
  subgroup~\(N_3\) isomorphic to~\(\Z^n\) in~\(H_3\), which is
  mapped by~\(\tcm_3\) onto a discrete subgroup in~\(G_3\).  Using
  Example~\ref{exa:cm-equivalence_quotient} once again, we find
  that~\(\cm_3\) is equivalent to~\(\cm_4\) with \(G_4=G_3/N_3\cong
  \Torus^n\), \(H_4=H_3/N_3\) and \(\tcm_4\) and~\(\acm_4\) induced
  by \(\tcm_3\) and~\(\acm_3\).  The quotient~\(H_4\) is still
  discrete, and mapped by~\(\tcm_4\) onto a dense subgroup in the
  torus~\(G_4\).  Thus~\(\cm_4\) has the desired
  form.

  Finally, we show that crossed modules with \(G_i\cong \R^{n_i}\)
  and discrete~\(H_i\) are only equivalent when they are isomorphic
  through some invertible linear map \(G_1\to G_2\) that maps
  \(\tcm(H_1)\) onto~\(\tcm(H_2)\).

  We observe first that the quotient Lie algebra
  \(\Lag/\Lah\) for a crossed module~\(\cm\) is
  invariant under equivalence.  Any homomorphism of crossed modules
  \(\cm_1\to\cm_2\) induces a pair of Lie algebra homomorphisms
  \(\Lah_1\to\Lah_2\) and
  \(\Lag_1\to\Lag_2\) that intertwine the
  differentials of \(\tcm_1\) and~\(\tcm_2\) and hence induce a
  homomorphism \(\Lag_1/\Lah_1
  \to\Lag_2/\Lah_2\).  It is not hard to see that
  this map is invertible if the homomorphism is an equivalence.
  Roughly speaking, \(\Lag/\Lah\) is the tangent
  space at the unit element in~\(\pi_1(\cm)\), and it is invariant
  because~\(\pi_1(\cm)\) as a topological group is invariant.

  If \(X\in\Lag/\Lah\) and one representative
  \(\hat{X}\in\Lag\) has the property that \(\exp(\hat{X})\in
  \tcm(H)\), then this holds for all representatives.  Let us denote
  this subset of \(\Lag/\Lah\) by \(T(\cm)\).  A
  crossed module homomorphism \(\cm_1\to\cm_2\) maps \(T(\cm_1)\to
  T(\cm_2)\), and if it is an equivalence, then it maps \(T(\cm_1)\)
  onto \(T(\cm_2)\) because it induces a bijection between
  \(G_i/\tcm_i(H_i)\) for \(i=1,2\).  Hence the pair consisting of the
  Lie algebra \(\Lag/\Lah\) and the subset \(T(\cm)\)
  is invariant under equivalence of crossed modules (in the sense that
  equivalent crossed modules have canonically isomorphic invariants).

  If a crossed module has discrete~\(H\) and \(G=\R^n\), then we
  identify \(\Lag/\Lah=\R^n\) in the obvious way and
  find that the subset of~\(X\) with \(\exp(\hat{X})\in \tcm(H)\) is
  precisely~\(\tcm(H)\).  This determines our crossed module
  because~\(\tcm\) is injective, \(\acm\) is trivial, and~\(H\) is
  discrete.  Hence crossed modules of this form are only equivalent
  when they are isomorphic.
\end{proof}

Given a thin crossed module \(\cm=(G,H,\tcm,\acm)\), the proof of
Theorem~\ref{the:thin_Lie_equivalent_Abelian} shows that
\(\Lag/\Lah\) is the Abelian Lie algebra~\(\R^n\)
for some \(n\in\N\), that \(T(\cm)\subseteq
\Lag/\Lah\) is a dense subgroup, and that~\(\cm\) is
equivalent to the thin Abelian crossed module~\(\cm'\) with \(G'=
\Lag/\Lah\), \(H'=T(\cm)\) with the discrete
topology, \(\tcm'\) the inclusion map, and trivial~\(\acm'\).  Hence
we get an explicit Abelian replacement for~\(\cm\) and do not have
to follow the steps in the above proof to construct it.

\begin{example}
  \label{exa:R_in_Tsquare}
  Consider the thin crossed module from a dense embedding
  \(\R\subset\Torus^2\) given by a line of irrational slope.  This
  is not yet in standard form because~\(\R\) is not discrete.  Of
  course, the image of the connected component of the identity is
  not closed in this example.  But when we pass to the universal
  covering~\(\R^2\) of~\(\Torus^2\), then we get the equivalent
  crossed module \(\R+\Z^2\subset \R^2\).  Now the image of the
  connected component is a line of irrational slope, which is closed
  in~\(\R^2\), and we may divide it out to get an equivalent crossed
  module \(\Z^2\to \R\).  This is now in standard form, and indeed
  of the expected form \(G'= \Lag/\Lah\),
  \(H'=T(\cm)\).
\end{example}

\begin{example}
  \label{exa:disretise}
  Let \(G=\Torus\), and let \(H=G\) with the discrete topology.  This
  is an example of a crossed module which is already in the second
  standard form of Theorem~\ref{the:thin_Lie_equivalent_Abelian}.
  Theorem~\ref{the:thin_Lie_equivalent_Abelian} works also if our
  Lie groups have uncountably many components.
\end{example}

The standard form with \(G=\Torus^n\) in
Theorem~\ref{the:thin_Lie_equivalent_Abelian} is sometimes less
useful because it is less unique, but it has the advantage
that~\(G\) is compact.

Now we turn to general locally compact groups.  Here a complete
classification seems hopeless, and even Abelianness fails in general
(see Example~\ref{exa:noncommutative_pione} below).  The following
theorem is what we can prove:

\begin{theorem}
  \label{the:thin_locally_compact_normalise}
  Any thin crossed module of locally compact groups is equivalent to
  a thin crossed module \(\cm'=(G',H',\tcm',\acm')\) with
  compact~\(G'\) and discrete~\(H'\).

  The following are equivalent:
  \begin{enumerate}
  \item \(\cm\) is equivalent to a thin crossed module~\(\cm'\) with
    compact Abelian~\(G'\), discrete~\(H'\), and trivial~\(\acm'\);
  \item \(\cm\) is equivalent to a crossed module~\(\cm'\) with
    trivial~\(\acm'\);
  \item the commutator map \(G\times G\to G\), \((x,y)\mapsto
    xyx^{-1}y^{-1}\), factors as \(\tcm\circ\gamma\) for a
    continuous map \(\gamma\colon G\times G\to H\).
  \end{enumerate}
\end{theorem}

\begin{proof}
  We construct a finite chain of crossed modules
  \(\cm_i=(G_i,H_i,\tcm_i,\acm_i)\) equivalent to~\(\cm\) with
  increasingly better properties.  Each step uses
  Example~\ref{exa:cm-equivalence_enlarge} or
  Example~\ref{exa:cm-equivalence_quotient}, which describe how the
  maps \(\tcm_i\) and~\(\acm_i\) and one of the groups in the next
  step are constructed.  We must only describe the subgroup to which
  we restrict in Example~\ref{exa:cm-equivalence_enlarge} or the
  normal subgroup we divide out in
  Example~\ref{exa:cm-equivalence_quotient}.

  By the structure theory of locally compact groups, the group~\(G\)
  contains an open, almost connected subgroup~\(G_1\).  Restricting to
  this subgroup as in Example~\ref{exa:cm-equivalence_enlarge}, we get
  an equivalent crossed module~\(\cm_1\) with almost
  connected~\(G_1\).

  Next we divide out suitable compact, \(\acm_1\)\nb-invariant (hence
  normal) subgroups in~\(H_1\).  On such subgroups, the map~\(\tcm_1\)
  is automatically a homeomorphism onto a compact subgroup of~\(G_3\),
  so that we may divide it out using
  Example~\ref{exa:cm-equivalence_quotient}.  Let~\(H_1^0\) be the
  connected component of the identity in~\(H_1\).  This closed
  subgroup is intrinsically defined and hence \(\acm_1\)\nb-invariant.
  Being connected, it contains a compact normal subgroup~\(N\) so
  that~\(H_1^0/N\) is a Lie group.  If \(N_1\) and~\(N_2\) are such
  subgroups, then so is \(N_1\cdot N_2\) because of normality.  Hence
  there is a maximal compact normal subgroup~\(N_1\) in~\(H_1^0\).
  This subgroup of~\(H_1^0\) is intrinsically defined and hence
  \(\acm_1\)\nb-invariant.  We may now pass to an equivalent quotient
  crossed module~\(\cm_2\) with \(H_2=H_1/N_1\).  The group~\(G_2\) is
  still almost connected, and \(H_2^0=H_1^0/N_1\) is now a Lie group.

  Since~\(G_2\) is almost connected, it contains a decreasing net of
  compact normal subgroups~\(K_i\) such that~\(G_2/K_i\) are Lie
  groups and \(\bigcap K_i=\{1\}\).  The map~\(\tcm_1\) induces Lie
  algebra homomorphisms from the Lie algebra~\(\Lah_2\)
  of~\(H_2^0\) to the Lie algebras of~\(G_2/K_i\).  Since~\(\tcm_1\)
  is injective and~\(\Lah_2\) is finite-dimensional, this map
  is injective for some~\(i\).  Thus there is a compact normal
  subgroup \(K\subseteq G_2\) such that~\(G_2/K\) is a Lie group and
  the map \(\Lah_2\to \Lag_2/\Lak\) is injective.

  Passing to a subgroup and a covering group \(G_3\to G_2\) as in
  the proof of Theorem~\ref{the:thin_Lie_equivalent_Abelian}, we now
  find an equivalent crossed module~\(\cm_3\) for which~\(G_3/K\) is
  simply connected.  The Lie algebras of \(H_2\) and~\(H_3\) are the
  same, so the map \(\Lah_3\to \Lag_3/\Lak\) remains
  injective.  Since~\(G_3/K\) is simply connected, there is a unique
  connected, closed, normal subgroup~\(L_3\) of~\(G_3/K\) whose Lie
  algebra is the image of~\(\Lah_3\) (see the proof of
  Theorem~\ref{the:thin_Lie_equivalent_Abelian}); a long exact
  sequence of homotopy groups using that~\(\pi_2\) vanishes for all
  Lie groups (such as \(G_3/KL_3\)) shows that~\(L_3\) is simply
  connected.  The map \(H_3^0\to L_3\) is the identity on Lie
  algebras and hence a covering map.  Since~\(L_3\) is already
  simply connected, it is a homeomorphism.  Then the map
  \(\tcm_3\colon H_3^0\to G_3\) is a homeomorphism onto its image as
  well.  Hence Example~\ref{exa:cm-equivalence_quotient} gives us an
  equivalent crossed module~\(\cm_4\) with \(H_4=H_3/H_3^0\).
  Now~\(H_4\) is totally disconnected, and~\(G_4\) is still almost
  connected.

  Since~\(H_4\) is totally disconnected, any action of a connected
  group on it is trivial.  In particular, the conjugation action
  of~\(G_4\) factors through~\(G_4/G_4^0\).  Since this group is
  compact, an intersection \(U'\defeq \bigcap_{g\in G_4} \acm_{g}(U)\)
  for an open subset \(U\subseteq H_4\) is again open.  If~\(U\) is a
  compact open subgroup, then this intersection is a compact, open,
  and \(\acm_4\)\nb-invariant subgroup in~\(H_4\).
  Example~\ref{exa:cm-equivalence_quotient} gives us an equivalent
  crossed module~\(\cm_5\) with \(H_5=H_4/U'\).  Now~\(H_5\) is
  discrete, and~\(G_5\) is still almost connected.

  Let \(g\in G_5^0\) belong to the connected component of the
  identity.  Since~\(G_5^0\) is connected and~\(H_5\) discrete,
  \(\acm_g(h)=h\) for all \(h\in H_5\).  Hence
  \(g\tcm_5(h)g^{-1}=\tcm_5(h)\) for all \(h\in H_5\).
  Since~\(\tcm_5(H)\) is dense in~\(G_5\), this shows that~\(g\) is
  central.  Hence the connected component~\(G_5^0\) is a central
  subgroup in~\(G_5\).

  The image of \(K_3\subseteq G_3\) in~\(G_5\) is a compact normal
  subgroup~\(K_5\) such that~\(G_5/K_5\) is connected, so
  that~\(G_5^0\) surjects onto it.  Since~\(G_5^0\) is central, the
  quotient group~\(G_5/K_5\) is a commutative Lie group.  As in the
  proof of Theorem~\ref{the:thin_Lie_equivalent_Abelian}, we now
  find a subgroup \(N_5\subseteq H_5\) isomorphic to~\(\Z^l\) whose
  image in~\(G_5/K_5\) is a lattice.  Dividing out this subgroup as
  in Example~\ref{exa:cm-equivalence_quotient}, we find an
  equivalent crossed module~\(\cm_6\) where~\(G_6\) is compact
  and~\(H_6\) is still discrete.  This proves our first statement.

  Assume now that the commutator map in~\(G\) factors through a
  continuous map \(G\times G\to H\); since~\(\tcm\) is injective,
  this just means that it is a map to~\(H\), and continuous as such.
  The crossed modules~\(\cm_i\) constructed above inherit this
  property in each step.  Thus the commutator map \(\gamma\colon
  G_6\times G_6\to G_6\) is a continuous map to~\(H_6\).
  Since~\(H_6\) is discrete, there is an open subset \(U\subseteq
  G_6\) with \(\gamma(x,y)=1\) for all \(x,y\in U\), that is,
  \(xy=yx\) for all \(x,y\in U\).  This remains so for~\(x,y\) in
  the subgroup~\(G'\) generated by~\(U\), which is open in~\(G_6\)
  because~\(U\) is open.  Example~\ref{exa:cm-equivalence_enlarge}
  shows that~\(G'\) is part of a crossed module equivalent
  to~\(\cm\).  This has Abelian compact~\(G'\), discrete~\(H'\), and
  hence trivial~\(\acm'\).

  As a consequence, \ref{the:thin_locally_compact_normalise}.3
  implies~\ref{the:thin_locally_compact_normalise}.1.  It is trivial
  that~\ref{the:thin_locally_compact_normalise}.1
  implies~\ref{the:thin_locally_compact_normalise}.2.  It remains to
  show that~\ref{the:thin_locally_compact_normalise}.3 follows,
  conversely, if~\(\cm\) is equivalent to a crossed module~\(\cm'\)
  with trivial~\(\acm'\).  Since~\(\cm'\) is again thin, this implies
  that \(G'\) and hence \(\coker\tcm'=\pi_1(\cm')\) is Abelian.  Then
  \(\pi_1(\cm) = G/\tcm(H)\) is Abelian as well because~\(\pi_1\) is
  invariant under equivalence of crossed modules.  This says that the
  commutator map of~\(G\) factors through~\(H\); it does not yet give
  the continuity of the factorisation.  We merely sketch how to prove
  this continuity.  The commutator map of~\(G'\) is constant and hence
  clearly a continuous map to~\(H'\).  To finish the proof, we must
  show that the existence of a continuous commutator map \(G\times
  G\to H\) for thin crossed modules is invariant under equivalence of
  crossed modules.

  We turn the functor \(M\colon (G\rtimes H)\times (G\rtimes H)\to
  (G\rtimes H)\) on the arrow groupoids
  in~\eqref{eq:mult_arrow_groupoid} into a generalised morphism
  (bispace); two functors give isomorphic bispaces if and only if
  they are related by conjugation with a bisection.  For a thin
  crossed module, a bisection that conjugates~\(M\) onto
  \(M\circ\mathrm{flip}\) is exactly the same as a continuous
  commutator map \(G\times G\to H\).  Thus the property of having a
  continuous commutator map \(G\times G\to H\) is equivalent to the
  property that \(M\) and \(M\circ\mathrm{flip}\) are equivalent as
  generalised morphisms.  This property is manifestly invariant
  under equivalence of crossed modules.
\end{proof}

\begin{lemma}
  \label{lem:thin_cm_commutative_sufficient}
  Let \(\cm=(G,H,\tcm,\acm)\) be a thin crossed module with a
  compact subset \(K\subseteq H\) such that~\(\cl{\tcm(K)}\)
  generates an open subgroup in~\(G\).  Then~\(\cm\) is equivalent
  to a crossed module~\(\cm'\) with trivial~\(\acm'\).
\end{lemma}

\begin{proof}
  Call a subgroup of a locally compact group \emph{compactly
    generated} if it is generated by a compact subset.  Our
  assumption is that there is a compactly generated subgroup~\(A\)
  of~\(H\) for which~\(\cl{\tcm(A)}\) is open in~\(G\).  If
  \(G_1\subseteq G\) is an open subgroup, then \(A_1\defeq
  \tcm^{-1}(G_1)\cap A\) has the same property for the crossed
  module~\(\cm_1\) constructed as in
  Example~\ref{exa:cm-equivalence_enlarge}.  If \(\pi\colon G_2\to
  G\) is a quotient mapping for which \(\ker\pi\) is compactly
  generated, then~\(\cm_2\) constructed as in
  Example~\ref{exa:cm-equivalence_quotient} also inherits this
  property, taking the preimage of~\(A\) in~\(A_2\), which is again
  finitely generated.  The coverings needed in the proof of
  Theorem~\ref{the:thin_locally_compact_normalise} have compactly
  generated kernels because we divide out either compact groups,
  connected Lie groups, or discrete subgroups in connected Lie
  groups, which are all compactly generated.  Hence
  Theorem~\ref{the:thin_locally_compact_normalise} provides an
  equivalent crossed module~\(\cm'\) with compact~\(G'\) and
  discrete~\(H'\) and a compactly generated subgroup \(A\subseteq
  H'\) for which \(\cl{\tcm(A)}\) is dense in~\(G'\).  Since~\(H'\)
  is discrete, \(A\) is generated by a finite subset~\(S\).
  Since~\(\acm'\) is continuous, the set of \(g\in G'\) with
  \(gh=hg\) for all \(h\in S\) is open in~\(G'\).  Since~\(S\)
  generates~\(G'\) topologically, this open subgroup is central
  in~\(G'\).  Thus~\(G'\) has an open, finite-index centre~\(Z\).
  Now replace~\(\cm'\) by an equivalent crossed module~\(\cm''\) as
  in Example~\ref{exa:cm-equivalence_enlarge} with \(G''=Z\).  This
  has commutative~\(G''\) and hence trivial~\(\acm''\)
  because~\(\tcm\) is injective.
\end{proof}

\begin{example}
  \label{exa:noncommutative_pione}
  We construct a thin crossed module of locally
  compact groups where~\(G/\tcm(H)\) is not Abelian.  Since
  \(G/\tcm(H)=\coker\tcm\) is invariant under equivalence, this
  crossed module cannot be equivalent to a commutative one.
  Let~\(G_1\) be some finite group that is not Abelian.  Let
  \[
  H \defeq \bigoplus_{n\in\N} G_1,\qquad
  G \defeq \prod_{n\in\N} G_1,
  \]
  where \(\bigoplus G_1\) is the subgroup of all \((g_n)\in \prod
  G_1\) with \(g_n=1\) for all but finitely many entries.
  Hence~\(H\) is a countable group; we give it the discrete
  topology.  The group~\(G\) is pro-finite and in particular
  compact.  It contains~\(H\) as a dense normal subgroup, and the
  conjugation map \(g\mapsto ghg^{-1}\) for any fixed \(h\in H\)
  factors through a finite product and therefore is continuous.  The
  constant embedding \(G_1\to G\) remains an embedding into~\(G/H\).
  Hence \(G/H\) is not Abelian.
\end{example}

Some of our simplifications also work for crossed modules that are not
thin.  We give one such statement:

\begin{proposition}
  \label{pro:cm_connected_Lie}
  Let \(\cm=(G,H,\tcm,\acm)\) be a crossed module of Lie groups with
  connected~\(G\) and injective~\(\tcm\).  Then~\(\cm\) is
  equivalent to a crossed module \(\cm'=(G',H',\tcm',\acm')\)
  where~\(G'\) is simply connected, \(H'\) is discrete,
  and~\(\acm'\) is trivial.  Two crossed modules of this form are
  equivalent if and only if they are isomorphic.
\end{proposition}

\begin{proof}
  As in the proof of Theorem~\ref{the:equivalence_for_actions}, we
  may pass to an equivalent crossed module~\(\cm_1\) where~\(G_1\)
  is the universal covering of~\(G\); in this new crossed module,
  the image of the connected component of~\(H_1\) is closed.
  Since~\(\tcm\) is assumed injective, we may divide out this
  connected component and arrive at an equivalent crossed module
  with discrete~\(H_2\) and simply connected~\(G_2\).  As in the
  proof of Theorem~\ref{the:equivalence_for_actions}, it follows
  that the conjugation action on~\(H_2\) is trivial.
  Hence~\(\cm_2\) has the asserted properties.

  Now consider the invariant \((\Lag/\Lah,T(\cm))\)
  under equivalence used already in the proof of
  Theorem~\ref{the:equivalence_for_actions}.  If \(\cm_1\)
  and~\(\cm_2\) are two crossed modules with discrete~\(H\), simply
  connected~\(G\) and injective~\(\tcm\), then an isomorphism
  \(\Lag_1/\Lah_1\to\Lag_2/\Lah_2\)
  is an isomorphism \(\Lag_1\to\Lag_2\) because
  \(H_1\) and~\(H_2\) are discrete.  This lifts to an isomorphism
  \(G_1\to G_2\) because \(G_1\) and~\(G_2\) are simply connected.
  Furthermore, we claim that the exponential maps
  \(\Lag_i\to G_i\) map \(T(\cm_i)\) onto~\(\tcm(H_i)\).
  This is because \(\tcm(H_i)\) is contained in the centre
  of~\(G_i\), the centre of~\(G_i\) is isomorphic to~\(\R^n\) for
  some~\(n\) because~\(G_i\) is simply connected, and the
  exponential map for~\(\R^n\) is surjective.  Therefore, an
  isomorphism between the invariants
  \((\Lag_1/\Lah_1,T(\cm_1))\) and
  \((\Lag_2/\Lah_2,T(\cm_2))\) already implies that
  \(\cm_1\) and~\(\cm_2\) are isomorphic crossed modules.
\end{proof}

\section{Duality for Abelian crossed modules}
\label{sec:crossed_acm_trivial}

\begin{definition}
  \label{def:two-Abelian}
  We call a crossed module \((G,H,\tcm,\acm)\) \emph{\(2\)\nb-Abelian} if
  the action~\(\acm\) is trivial, and \emph{Abelian} if~\(\acm\) is
  trivial and~\(G\) is Abelian.
\end{definition}

For a \(2\)\nb-Abelian crossed module, \(H\) is Abelian because
\(hkh^{-1}=\acm_{\tcm(h)}(k)\) for all \(h,k\in H\), and \(\tcm(H)\)
is a central subgroup of~\(G\) because
\(\tcm(\acm_g(h))=g\tcm(h)g^{-1}\).  For a thin crossed module,
\(H\)~is Abelian if and only if~\(\acm\) is trivial, if and only
if~\(G\) is Abelian. Hence, for thin crossed modules, Abelianness
and \(2\)\nb-Abelianness are equivalent conditions.  Of course, in
general a \(2\)\nb-Abelian crossed module need not be Abelian (for
instance, any group~\(G\) viewed as a crossed module \((G,0,0,0)\)
is \(2\)\nb-Abelian).  Proposition~\ref{pro:cm_connected_Lie}
implies that any crossed module of Lie groups with connected~\(G\)
and injective~\(\tcm\) is equivalent to a \(2\)\nb-Abelian one.

Let \(\cm=(G,H,\tcm,\acm)\) be a \(2\)\nb-Abelian crossed module of
locally compact groups.  Let \((\A,\u)\) be a Fell bundle
over~\(\cm\).  Condition~\ref{def:FellBundle2Group2} in
Definition~\ref{def:FellBundle2Group} now says that \(a\cdot
\u_h=\u_h\cdot a\) for all \(a\in\A\), \(h\in H\).  This extends to
\(a\in\contc(\A)\) and then to \(a\in\Cst(\A)\); here we
embed~\(\u_h\) into the multiplier algebra of~\(\Cst(\A)\) using the
universal representation~\(\rho^\un\).  Thus~\(\rho^\un\) now
maps~\(H\) to the centre of \(\Mult(\Cst(\A))\).  This map
integrates to a nondegenerate \Star{}homomorphism
\begin{equation}
  \label{eq:hatH_structure}
  {\smallint}\rho^\un\colon \Cst(H)\to Z\Mult(\Cst(\A)).
\end{equation}
Identifying \(\Cst(H)\) with \(\Cont_0(\hat{H})\) for the Pontryagin
dual~\(\hat{H}\) of~\(H\), we get a structure of
\(\Cont_0(\hat{H})\)\nb-algebra on~\(\Cst(\A)\).  (We normalise the
Fourier transform so that the unitary
\(\delta_h\in\U\Mult(\Cst(H))\) becomes the function
\(\hat{h}\mapsto \hat{h}(h)\) on~\(\hat{H}\).)

\begin{proposition}
  \label{pro:crossed_acm_trivial}
  The crossed product~\(\Cst(\A,\u)\) is the fibre at \(1\in\hat{H}\)
  for the \(\Cont_0(\hat{H})\)-\(\Cst\)-algebra structure
  on~\(\Cst(\A)\) just described.
\end{proposition}

\begin{proof}
  By construction, \(\Cst(\A,\u)\) is the quotient by the relation
  \(\rho^\un(\u_h)=1\) for all \(h\in H\).  This is equivalent to
  \({\smallint}\rho^\un(f)= f(1_{\hat{H}})\) for all
  \(f\in\Cont_0(\hat{H})\cong \Cst(H)\).  Hence \(\Cst(\A,\u)\) is the
  quotient by \(\Cont_0(\hat{H}\setminus \{1\})\cdot\Cst(\A)\), that
  is, the fibre at~\(1\).
\end{proof}

The above proposition is merely an observation, the main point is to
see that \(\Cst(\A)\) has a relevant \(\Cont_0(\hat{H})\)-algebra
structure.  This proposition allows us to compute crossed products by
\(2\)\nb-Abelian crossed modules in two more elementary steps: first take
the crossed product~\(\Cst(\A)\) by the locally compact group~\(G\);
then take the fibre at \(1\in\hat{H}\) for the canonical
\(\Cont_0(\hat{H})\)-algebra structure on~\(\Cst(\A)\).

\begin{example}
  \label{exa:crossed_Z_T}
  Let \(\theta\in\R\setminus\Q\) and let~\(\Rotc{\theta}\) be the
  associated noncommutative torus.  It carries a strict action by the
  Abelian crossed module \(\tcm\colon \Z\to\Torus\) with
  \(\tcm(n)\defeq \exp(2\pi \ima \theta n)\); namely, \(\Torus\) acts by part
  of the gauge action: \(\alpha_z(U)\defeq zU\), \(\alpha_z(V)\defeq
  V\), where \(U\) and~\(V\) are the standard generators
  of~\(\Rotc{\theta}\), and~\(\Z\) acts by \(n\mapsto V^{-n}\)
  (see~\cite{Buss-Meyer-Zhu:Non-Hausdorff_symmetries}).  The crossed
  product for this action is already computed
  in~\cite{Buss-Meyer-Zhu:Non-Hausdorff_symmetries}: it is the
  \(\Cst\)\nb-algebra of compact operators on the Hilbert
  space~\(L^2(\Torus)\).  This also follows from
  Proposition~\ref{pro:crossed_acm_trivial}.  First, the crossed
  product \(\Rotc{\theta}\rtimes\Torus\) turns out to be
  \(\Cont(\Torus)\otimes \Comp(L^2\Torus)\) because the
  \(\Torus\)\nb-action on~\(\Rotc{\theta}\) is a dual action if we
  interpret \(\Rotc{\theta}\) as \(\Cont(\Torus)\rtimes\Z\) with
  \(\Cont(\Torus)=\Cst(V^*)\) and~\(U\) generating~\(\Z\).  The map
  \(\Cst(\Z)\to \Rotc{\theta}\rtimes\Torus\) becomes an isomorphism
  onto \(\Cont(\Torus)\otimes 1\).  Hence the fibre at
  \(1\in\hat{\Z}=\Torus\) gives \(\Comp(L^2\Torus)\) as expected.
\end{example}

Next we generalise Takesaki--Takai duality from Abelian groups to
Abelian crossed modules. We first reformulate the classical statement
in our setting of correspondence \(2\)\nb-categories.

Let~\(G\) be an Abelian locally compact group and let~\(\hat{G}\) be
its Pontryagin dual.  A crossed product for a \(G\)\nb-action carries
a canonical dual action of~\(\hat{G}\) and a crossed product for a
\(\hat{G}\)\nb-action carries a canonical dual action of~\(G\).  These
constructions extend to map \(G\)\nb-equivariant correspondences to
\(\hat{G}\)\nb-equivariant correspondences, and vice versa, such that
equivariant isomorphism of correspondences is preserved.  That is,
they provide functors \(\Corr(G)\leftrightarrow \Corr(\hat{G})\).
Takesaki--Takai duality implies that these functors are equivalences
of \(2\)\nb-categories inverse to each other.  Going back and forth gives
equivariant isomorphisms \(A\rtimes G\rtimes\hat{G} \cong A\otimes
\Comp(L^2G)\) and \(A\rtimes \hat{G}\rtimes G \cong A\otimes
\Comp(L^2\hat{G})\), respectively; we interpret these as equivariant
Morita--Rieffel equivalences \(A\rtimes \hat{G}\rtimes G \simeq A\)
and \(A\rtimes G\rtimes\hat{G} \simeq A\) via \(A\otimes L^2G\) and
\(A\otimes L^2\hat{G}\), respectively.  These equivariant
Morita--Rieffel equivalences are natural with respect to equivariant
correspondences, hence the functors we get by composing the crossed
product functors are naturally equivalent to the identity functors.

Now let \(\cm=(G,H,\tcm,0)\) be an Abelian crossed module of locally
compact groups; this is nothing but a continuous homomorphism between
two locally compact Abelian groups.  Its \emph{dual crossed
  module}~\(\hat{\cm}\) consists of the dual groups \(\hat{G}\)
and~\(\hat{H}\) and the transpose \(\hat{\tcm}\colon \hat{H}\to\hat{G}\)
of~\(\tcm\).  The bidual of~\(\cm\) is naturally isomorphic to~\(\cm\)
because this holds for locally compact groups.  Our duality theorem
does not compare the action \(2\)\nb-categories of the crossed modules
\(\cm\) and~\(\hat{\cm}\), however; instead, the action \(2\)\nb-category
of the arrow groupoid \(\hat{H}\rtimes\hat{G}\) of~\(\hat{\cm}\)
appears.  And the functors going back and forth are the crossed
product functors for the groups \(G\) and~\(\hat{G}\).

We now describe the equivalence of \(2\)\nb-categories \(\Corr(\cm)
\simeq \Corr(\hat{H}\rtimes\hat{G})\).  Let \((\A,\u)\) be a
saturated Fell bundle over~\(\cm\), encoding a \(\cm\)\nb-action by
correspondences.  The cross-sectional
\(\Cst\)\nb-algebra~\(\Cst(\A)\) carries a dual action
of~\(\hat{G}\), where a character acts by pointwise multiplication
on sections.  Since Abelian crossed modules are \(2\)\nb-Abelian,
\eqref{eq:hatH_structure} provides a nondegenerate
\Star{}homomorphism
\[
{\smallint}\rho^\un\colon
\Cont_0(\hat{H})\cong \Cst(H)\to Z\Mult(\Cst\A).
\]
Since~\(\upsilon_h\) is a multiplier of degree~\(\tcm(h)\), the dual
action~\(\alpha\) of~\(\hat{G}\) acts on \(\delta_h\in\Cst(H)\) by
\(\alpha_{\hat{g}}(\delta_h) = \hat{g}(\tcm h)\cdot \delta_h\).
This defines an action of~\(\hat{G}\) on \(\Cst(H)\), which
corresponds to the action \(\rho_{\hat{g}}f(\hat{h}) =
f(\hat{h}\cdot\hat{\tcm}(\hat{g}))\) on \(\Cont_0(\hat{H})\)
coming from the right translation action of~\(\hat{G}\)
on~\(\hat{H}\) through the homomorphism \(\hat{\tcm}\colon
\hat{G}\to\hat{H}\).  Thus \(\Cst(\A)\) carries a strict, continuous
action of the arrow groupoid \(\hat{H}\rtimes\hat{G}\)
of~\(\hat{\cm}\).

The actions of \(\hat{H}\rtimes\hat{G}\) form a \(2\)\nb-category
\(\Corr(\hat{H}\rtimes\hat{G})\) with actions by correspondences as
objects, equivariant correspondences as arrows and isomorphisms of
equivariant correspondences as \(2\)\nb-arrows.  We claim that the
construction on the level of objects above is part of a functor
\(\Corr(\cm)\to\Corr(\hat{H}\rtimes\hat{G})\).

Let~\(\Hilm\) be a correspondence of Fell bundles
\((\A,\u^\A)\to(\B,\u^\B)\).  The space \(\contc(\Hilm)\) of compactly
supported, continuous sections of~\(\Hilm\) is a pre-Hilbert module
over the \Star{}algebra \(\contc(\B)\) with a nondegenerate left
action of \(\contc(\A)\).  Using the embedding
\(\contc(\B)\subseteq\Cst(\B)\), we may complete \(\contc(\Hilm)\) to
a Hilbert \(\Cst(\B)\)-module \(\Cst(\Hilm)\).  The left action of
\(\contc(\A)\) on \(\contc(\Hilm)\) induces a left action of
\(\Cst(\A)\) on \(\Cst(\Hilm)\), turning it into a correspondence
\(\Cst(\A)\to\Cst(\B)\).  Isomorphic correspondences of Fell bundles
clearly give isomorphic correspondences \(\Cst(\A)\to\Cst(\B)\).  Thus
taking cross-sectional \(\Cst\)\nb-algebras gives a functor
\(\Corr(\cm)\to\Corr\) (where \(\Corr=\Corr(\{1\})\) denotes the
\(2\)\nb-category of \cstar{}algebras with correspondences as their
morphisms); actually, since we did not use the \(H\)\nb-action, this
is just the crossed product functor \(\Corr(G)\to\Corr\), composed
with the forgetful functor \(\Corr(\cm)\to\Corr(G)\).

Pointwise multiplication by characters defines a
\(\hat{G}\)\nb-action on \(\contc(\Hilm)\).  This extends to the
completion \(\Cst(\Hilm)\), and turns it into a
\(\hat{G}\)\nb-equivariant correspondence \(\Cst(\A)\to\Cst(\B)\).
This \(\hat{G}\)\nb-action is natural for isomorphisms of Fell
bundle correspondences.  Furthermore, \(\Cst(\Hilm)\) is a
\(\Cst(H)\)-linear correspondence because \(\u^\A_h\cdot \xi =
\xi\cdot \u^\B_h\) for all \(h=\acm_g(h)\in H\) by
\ref{def:corr_Fell_bundle}.\ref{def:corr_Fell_bundle_central}.  Thus
taking the cross-sectional \(\Cst\)\nb-algebra (that is, the crossed product
by~\(G\)) gives a functor
\(\Corr(\cm)\to\Corr(\hat{H}\rtimes\hat{G})\).

\begin{theorem}
  \label{the:Abelian_cm_groupoid}
  Let~\(\cm\) be an Abelian crossed module.  The functor
  \(\Corr(\cm)\to\Corr(\hat{H}\rtimes\hat{G})\) just described is an
  equivalence of \(2\)\nb-categories.
\end{theorem}

\begin{proof}
  To prove the existence of a quasi-inverse functor
  \(\Corr(\hat{H}\rtimes\hat{G})\to\Corr(\cm)\), it suffices to
  construct it for strict actions by automorphisms because arbitrary
  actions by correspondences are equivalent to strict actions by
  automorphisms (see \cite{Buss-Meyer-Zhu:Higher_twisted}*{Theorem
    5.3}), where ``equivalent'' means ``isomorphic in the
  \(2\)\nb-category \(\Corr(\hat{H}\rtimes\hat{G})\).''

  Thus let~\(B\) be a \(\Cst\)\nb-algebra with a continuous action
  of \(\hat{H}\rtimes\hat{G}\) in the usual sense.  This consists of
  a strict action~\(\beta\) of the group~\(\hat{G}\) and a
  \(\hat{G}\)\nb-equivariant nondegenerate \Star{}homomorphism from
  \(\Cont_0(\hat{H})\) to~\(Z\Mult(B)\).  The crossed product
  \(B\rtimes\hat{G}\) carries a dual action~\(\hat{\beta}\)
  of~\(G\).  For \(h\in H\), define
  \(v_h\in\U\Mult(\Cont_0(\hat{H}))\) by \(v_h(\hat{h})=
  \hat{h}(h)^{-1}\) for all \(\hat{h}\in\hat{H}\).  The right
  translation action of~\(\hat{G}\) acts on~\(v_h\) by
  \(\hat{g}(v_h)= \hat{g}(\tcm(h))^{-1}\cdot v_h\) because
  \[
  \hat{g}(v_h)(\hat{h})
  = v_h(\hat{h}\cdot \hat{\tcm}\hat{g})
  = (\hat{\tcm}\hat{g})(h)^{-1}\hat{h}(h)^{-1}
  = \hat{g}(\tcm h)^{-1} v_h(\hat{h})
  \]
  for all \(\hat{h}\in\hat{H}\).  Now map~\(v_h\) to
  \(\U\Mult(B\rtimes\hat{G})\) using the homomorphism
  \(\U\Mult(\Cont_0(\hat{H}))\to \U\Mult(B\rtimes\hat{G})\) induced
  by the \(\Cont_0(\hat{H})\)-\(\Cst\)-algebra structure on~\(B\).
  These unitaries \(v_h\in\U\Mult(B\rtimes\hat{G})\) commute
  with~\(B\) because \(\Cont_0(\hat{H})\) is mapped to the centre
  of~\(B\), and they satisfy
  \[
  v_h\delta_{\hat{g}} v_h^*
  = \delta_{\hat{g}} \hat{g}(\tcm h)\cdot v_h v_h^*
  = \hat{g}(\tcm h)\cdot \delta_{\hat{g}}
  \]
  for \(\hat{g}\in\hat{G}\).  Hence \(\hat\beta_{\tcm(h)}=\Ad(v_h)\)
  for all \(h\in H\).  Since \(\hat{\beta}_g(v_h)=v_h\) for all
  \(g\in G\), \(h\in H\), the map \(h\mapsto v_h\) and the dual
  action~\(\hat{\beta}\) of~\(G\) combine to a strict action of the
  crossed module~\(\cm\) on \(B\rtimes\hat{G}\).

  The above constructions extend to equivariant correspondences in a
  natural way and thus provide a functor
  \(\Corr(\hat{H}\rtimes\hat{G})\to\Corr(\cm)\).  We claim that this
  is quasi-inverse to the functor
  \(\Corr(\cm)\to\Corr(\hat{H}\rtimes\hat{G})\) constructed above.
  We must compose these functors in either order and check that
  the resulting functors are equivalent to the identity
  functors.  Since all actions by correspondences are equivalent
  to strict actions by automorphisms, it is enough to verify the
  equivalence on those objects in \(\Corr(\cm)\) and
  \(\Corr(\hat{H}\rtimes\hat{G})\) that are strict actions by
  automorphisms.

  Let \(\alpha\colon G\to\Aut(A)\) and \(u\colon H\to\U\Mult(A)\) be
  a strict action of~\(\cm\) on~\(A\).  Our functor to
  \(\Corr(\hat{H}\rtimes\hat{G})\) maps it to \(A\rtimes G\)
  equipped with the dual action of~\(\hat{G}\) and with the
  canonical map \(\Cont_0(\hat{H})\cong \Cst(H)\to Z\Mult(A\rtimes
  G)\) induced by the representation \(h\mapsto
  u_h^*\delta_{\tcm(h)}\) for \(h\in H\); this is how the
  translation from strict actions to Fell bundles works on the level
  of the unitaries~\(u_h\) and~\(\u_h\) (see
  Example~\ref{exa:StrictToFellBundles}).  The map to actions
  of~\(\cm\) takes this to \((A\rtimes G)\rtimes\hat{G}\) equipped
  with the dual action of~\(G\) and the homomorphism
  \(H\to\U\Mult(A\rtimes G\rtimes\hat{G})\), \(h\mapsto
  (u_h^*\delta_{\tcm(h)})^{-1} \in \U\Mult(A\rtimes G)\subseteq
  \U\Mult(A\rtimes G\rtimes\hat{G})\) because the isomorphism
  \(\Cont_0(\hat{H})\cong \Cst(H)\) maps \(v_h\mapsto
  \delta_{h^{-1}}\).

  Takesaki--Takai duality provides a canonical \(G\)\nb-equivariant
  isomorphism
  \[
  A\rtimes G\rtimes\hat{G}\cong A\otimes \Comp(L^2G)
  \cong \Comp(A\otimes L^2G).
  \]
  It extends the standard representation of \(A\rtimes G\) on
  \(A\otimes L^2 G\) that maps \(a\in A\) to the operator
  \(\varphi(a)\) of pointwise multiplication by \(G\ni g\mapsto
  \alpha_g(a)\) and \(g\in G\) to the right translation operator
  \(\rho_gf(g')=f(g'g)\) for all \(g'\in G\).  Hence
  \((u_h^*\delta_{\tcm(h)})^{-1} = \delta_{\tcm(h)^{-1}} u_h\) acts
  by the unitary operator
  \[
  (\rho_{\tcm(h)^{-1}}\varphi(u_h) f)(g) = u_h\cdot f(g\tcm(h)^{-1})
  = u_h\cdot f(\tcm(h)^{-1}g)
  \]
  because \(G\) is Abelian and \(\alpha_g(u_h)=u_h\) for all \(g\in
  G\), \(h\in H\).  This gives the operator
  \(u_h\otimes\lambda_{\tcm(h)}\) for the left regular
  representation~\(\lambda\) on~\(G\).  The imprimitivity bimodule
  \(A\otimes L^2G\) between \(\Comp(A\otimes L^2G)\cong A\rtimes
  G\rtimes\hat{G}\) and~\(A\) with the \(G\)\nb-action
  \(g\mapsto\alpha_g\otimes\lambda_g\) is \(\cm\)\nb-equivariant and
  thus provides a \(\cm\)\nb-equivariant equivalence \(A\simeq
  A\rtimes G\rtimes\hat{G}\).

  Now let \(B\) carry a strict \(\hat{H}\rtimes\hat{G}\)-action.  That
  is, \(\hat{G}\) acts on~\(B\) via an action~\(\beta\), and we have a
  nondegenerate \(\hat{G}\)\nb-equivariant \Star{}homomorphism
  \(\phi\colon\Cont_0(\hat{H})\to Z\Mult(B)\).  Our functor to
  \(\Corr(\cm)\) takes this to \(B\rtimes \hat{G}\) equipped with the
  dual action of~\(G\) and the homomorphism \(H\to\U Z\Mult(B\rtimes
  \hat{G})\) defined as the composite of the representation
  \(v\colon H\to Z\Mult(B)\), \(h\mapsto \phi(v_h)\), with
  \(v_h(\hat{h}) \defeq \hat{h}(h)^{-1}\) and the canonical embedding
  \(B\to \Mult(B\rtimes\hat{G})\).  The functor that goes back to an
  \(\hat{H}\rtimes\hat{G}\)-action now gives \(B\rtimes \hat{G}\rtimes
  G\) with the dual action of~\(\hat{G}\) and with the nondegenerate
  \Star{}homomorphism \(\Cst(H)\to Z\Mult(B\rtimes \hat{G}\rtimes G)\)
  that maps \(h\in H\) to \(\phi(v_h^*)\delta_{\tcm(h)}\) -- here
  \(\phi(v_h)\) is viewed as an element of \(\Mult(B\rtimes
  \hat{G}\rtimes G)\) using the canonical homomorphism \(B\to
  \Mult(B\rtimes \hat{G}\rtimes G)\).  Now we identify
  \(\Cont_0(\hat{H})\cong\Cst(H)\) as before, mapping the
  character~\(v_h^*\) to~\(\delta_h\).  Thus the resulting map
  \(\Cont_0(\hat{H})\to Z\Mult(B\rtimes\hat{G}\rtimes G)\)
  maps~\(v_h^*\) to \(\phi(v_h^*)\delta_{\tcm(h)}\).  On the other
  hand, via the isomorphism \(\Cont_0(\hat{H})\cong\Cst(H)\), the
  homomorphism \(\phi\colon \contz(\hat{H})\to \Mult(B)\) corresponds
  to the integrated form of the representation \(H\to\Mult(B)\),
  \(h\mapsto \phi(v_h^*)\).  The \(\hat{G}\)\nb-equivariance
  of~\(\phi\) means that \(\beta_{\hat{g}}(\phi(v_h^*))
  =\hat{g}(\tcm(h))\cdot \phi(v_h^*)\) (recall that the
  \(\hat{G}\)\nb-action on \(\contz(\hat{H})\) is induced by the right
  translation \(\hat{G}\)\nb-action on~\(\hat{H}\) via
  \((\hat{h},\hat{g})\mapsto \hat{h}\cdot\hat{\tcm}(\hat{g})\)).

  Takesaki--Takai duality gives a \(\hat{G}\)\nb-equivariant isomorphism
  \[
  B\rtimes\hat{G}\rtimes G\cong B\otimes \Comp(L^2\hat{G})
  \cong \Comp(B\otimes L^2\hat{G}).
  \]
  It restricts to the standard representation of \(B\rtimes\hat{G}\)
  on \(B\otimes L^2\hat{G}\), where \(b\in B\) acts by \(b\cdot
  f(\hat{g}) = \beta_{\hat{g}}(b) f(\hat{g})\) and
  \(\hat{g}\in\hat{G}\) acts by right translation; the representation
  of~\(G\) lets \(g\in G\) act by \(g\cdot f(\hat{g})=\hat{g}(g)^{-1}
  f(\hat{g})\) for all \(f\in \contc(\hat{G},B)\sbe B\otimes
  L^2\hat{G}\).  Since \(\beta_{\hat{g}}(\phi(v_h^*))
  =\hat{g}(\tcm(h))\cdot \phi(v_h^*)\), we get
  \[
  (\phi(v_h^*)\delta_{\tcm(h)})\cdot f(\hat{g})
  = \hat{g}(\tcm(h))\cdot \phi(v_h^*) \cdot \hat{g}(\tcm(h))^{-1}\cdot f(\hat{g})
  = \phi(v_h^*)\cdot f(\hat{g}).
  \]
  This means that the \(\hat{G}\)\nb-equivariant
  \(B\rtimes\hat{G}\rtimes G\)-\(B\)-equivalence bimodule \(B\otimes
  L^2\hat{G}\) also intertwines the representations of~\(H\) (and
  hence the representations of \(\contz(\hat{H})\)) on
  \(B\rtimes\hat{G}\rtimes G\) and~\(B\).  In other words, \(B\otimes
  L^2\hat{G}\) is an \(\hat{H}\rtimes\hat{G}\)\nb-equivariant
  Morita--Rieffel equivalence between \(B\) and \(B\rtimes
  \hat{G}\rtimes G\).
\end{proof}

Despite Theorem~\ref{the:Abelian_cm_groupoid}, there is an important
difference between \(\Corr(\cm)\) and
\(\Corr(\hat{H}\rtimes\hat{G})\).  Both \(2\)\nb-categories come with a
natural tensor product structure: take the diagonal action on the
tensor product in \(\Corr(\cm)\), or the diagonal action on the tensor
product over the base space~\(\hat{H}\) in
\(\Corr(\hat{H}\rtimes\hat{G})\).  These tensor products are quite
different.  In terms of \(\Corr(\hat{H}\rtimes\hat{G})\), the natural
tensor product in \(\Corr(\cm)\) does the following.  Take two
\(\Cst\)\nb-algebras \(A_1\) and~\(A_2\) with actions of
\(\hat{H}\rtimes\hat{G}\).  Their tensor product is a
\(\Cst\)\nb-algebra over \(\hat{H}\times\hat{H}\) with a compatible
action of the group \(\hat{G}\times\hat{G}\).  Restrict the group
action to the diagonal (this gives the usual diagonal action).  But
instead of restricting the \(\Cst\)\nb-algebra to the diagonal in
\(\hat{H}\times\hat{H}\) as usual, give it a structure of
\(\Cont_0(\hat{H})\)-\(\Cst\)\nb-algebra using the comultiplication
\(\Cont_0(\hat{H})\to \Cont_0(\hat{H}\times\hat{H})\).

\begin{remark}
  \label{rem:duality_crossed_groupoid}
  There is a more symmetric form of our duality where both partners in
  the duality are of the same form: both are length-two chain
  complexes of locally compact Abelian groups
  \[
  H\xrightarrow{d} G\xrightarrow{d} K,
  \]
  with the dual of the form \(\hat{K}\to\hat{G}\to\hat{H}\).  An
  action of such a complex consists of an action of the crossed module
  \(H\to G\) and an action of the transformation groupoid \(K\rtimes G\), where both
  actions contain the same action of the group~\(G\); here the
  groupoid \(K\rtimes G\) is the transformation groupoid for the
  translation action \(k\cdot g\defeq k\cdot d(g)\) of~\(G\) on~\(K\).
  Actually, these actions are the actions of the \(2\)\nb-groupoid with
  object space~\(K\), arrows \(g\colon k\to k\cdot d(g)\), and
  \(2\)\nb-arrows \(H\colon g\Rightarrow g\cdot d(h)\); the chain complex
  condition \(d^2=0\) ensures that this is a strict \(2\)\nb-groupoid.  Our
  proof above shows that the \(2\)\nb-groupoids \(H\to G\to K\) and
  \(\hat{K}\to\hat{G}\to\hat{H}\) have equivalent action
  \(2\)\nb-categories on \(\Cst\)\nb-algebras.  Namely, an action of the
  crossed module from \(H\to G\) on~\(A\) induces an action of the
  groupoid from \(\hat{G}\to\hat{H}\) on \(A\rtimes G\), and an action
  of the groupoid from \(G\to K\) induces an action of the crossed
  module from \(\hat{K}\to\hat{G}\) on \(A\rtimes G\).  Since both
  constructions involve the same dual action of~\(\hat{G}\), we get an
  action of \(\hat{K}\to\hat{G}\to\hat{H}\) on \(A\rtimes G\).  Now
  the functor backwards has exactly the same form, and going back and
  forth is still a stabilisation functor.  Since the stabilisation is
  compatible with the crossed module and groupoid parts of our
  actions, it is compatible with their combination to a \(2\)\nb-groupoid
  action.
\end{remark}

In Section~\ref{sec:equivalence_crossed_modules} we have studied
equivalence of general crossed modules via homomorphisms.  We end
this section by a criterion for equivalences between Abelian crossed
modules.

\begin{proposition}
  Let \(\cm_i=(G_i,H_i,\tcm_i,0)\), \(i=1,2\), be Abelian crossed
  modules and let \((\varphi,\psi)\colon \cm_1\to\cm_2\) be a
  homomorphism.  The following are equivalent:
  \begin{enumerate}
  \item the homomorphism \((\varphi,\psi)\) is an equivalence;
  \item the diagram
    \[
    H_1 \xrightarrow{\iota_1}
    G_1\times H_2 \xrightarrow{\pi_2}
    G_2
    \]
    is an extension of locally compact Abelian groups, where
    \begin{alignat*}{2}
      \iota_1\defeq (\tcm_1^{-1},\psi)&\colon H_1\to G_1\times H_2,&\qquad
      h_1&\mapsto (\tcm_1(h_1)^{-1},\psi(h_1)),\\
      \pi_2&\colon G_1\times H_2\to G_2,&\qquad
      (g_1,h_2)&\mapsto \varphi(g_1)\tcm_2(h_2);
    \end{alignat*}
  \item the dual diagram
    \[
    \widehat{H_1} \xleftarrow{\widehat{\iota_1}}
    \widehat{G_1}\times \widehat{H_2} \xleftarrow{\widehat{\pi_2}}
    \widehat{G_2}
    \]
    is an extension of locally compact Abelian groups;
  \item the dual homomorphism \((\hat{\psi},\hat{\varphi})\colon
    \hat{\cm}_2\to \hat{\cm}_1\) is an equivalence.
\end{enumerate}
\end{proposition}

\begin{proof}
  We shall use Lemma~\ref{lem:cm_equivalence_criterion}.  Since
  inversion on a topological group is a homeomorphism, the map in
  condition~\ref{en:cm_equivalence_1} in
  Lemma~\ref{lem:cm_equivalence_criterion} is a homeomorphism if and
  only if~\(\iota_1\) is a homeomorphism onto \(\ker\pi_2\).  The
  map~\(\pi_2\) is the same one appearing in condition~\ref{en:cm_equivalence_2} in
  Lemma~\ref{lem:cm_equivalence_criterion}.  This is a homomorphism
  since all groups involved are Abelian, and its kernel is exactly
  the image of~\(\iota_1\).  Therefore, the assumptions in
  Lemma~\ref{lem:cm_equivalence_criterion} hold if and only if \(H_1
  \to G_1\times H_2 \to G_2\) is a topological group extension.
  Since taking duals preserves this property, this is equivalent to
  \(\widehat{H_1} \leftarrow \widehat{G_1}\times \widehat{H_2}
  \leftarrow \widehat{G_2}\) being a topological group extension.
  As above, this is equivalent to \((\hat{\psi},\hat{\varphi})\)
  being an equivalence.
\end{proof}

\section{Crossed products for crossed module extensions}
\label{sec:partial_crossed}

Now we come to the factorisation of the crossed product functor for
an extension of crossed modules.  Let
\(\cm_i=(G_i,H_i,\tcm_i,\acm_i)\) for \(i=1,2,3\) be crossed modules
of locally compact groups.

\begin{definition}
  \label{def:strict_extension}
  A diagram \(\cm_1\to\cm_2\to\cm_3\) of homomorphisms of crossed
  modules is called a \emph{strict extension} of crossed modules if
  the resulting diagrams
  \[
  H_1\xrightarrow{\psi_1} H_2\xrightarrow{\psi_2} H_3
  \quad\text{and}\quad
  G_1\xrightarrow{\varphi_1} G_2\xrightarrow{\varphi_2} G_3
  \]
  are extensions of locally compact groups.  That is, \(\psi_1\) is
  a homeomorphism onto the kernel of~\(\psi_2\) and~\(\psi_2\) is an
  open surjection, and similarly for \(\varphi_1\)
  and~\(\varphi_2\).
\end{definition}

\begin{theorem}
  \label{the:partial_crossed}
  Let \(\cm_1\into \cm_2\onto \cm_3\) be a strict extension of
  crossed modules and let~\(A\) be a \(\Cst\)-algebra with an action
  of~\(\cm_2\) by correspondences.  Then \(A\rtimes\cm_1\) carries a
  canonical action of~\(\cm_3\) by correspondences such that
  \((A\rtimes\cm_1)\rtimes \cm_3\) is naturally isomorphic
  to~\(A\rtimes\cm_2\).
\end{theorem}

\begin{proof}
  Since strict actions by automorphisms are notationally simpler, we
  first prove the result in case~\(A\) carries a strict action by
  automorphisms.  Then we reduce the general case to this special
  case.  A strict action by automorphisms is given by group
  homomorphisms \(\alpha\colon G_2\to\Aut(A)\) and \(u\colon
  H_2\to\U\Mult(A)\) that satisfy \(\alpha_{\tcm_2(h)}=\Ad_{u_h}\) for
  all \(h\in H_2\) and \(\alpha_g(u_h)=u_{\acm_g(h)}\) for all \(g\in
  G_2\), \(h\in H_2\).  To simplify notation, we also view \(G_1\)
  and~\(H_1\) as subgroups of \(G_2\) and~\(H_2\), respectively, so
  that we drop the maps \(\varphi_1\) and~\(\psi_1\).

  The crossed product \(A\rtimes\cm_1\) is a quotient of the crossed
  product \(A\rtimes G_1\) for the group~\(G_1\) by the ideal
  generated by the relation \(u_h\sim\delta_{\tcm_1(h)}\) for all
  \(h\in H_1\); that is, we divide \(A\rtimes G_1\) by the closed
  linear span of the subset
  \[
  \{x\cdot (u_h-\delta_{\tcm_1(h)})\cdot y:
  h\in H_1, x,y\in A\rtimes G_1\}.
  \]
  A canonical action~\(\gamma'\) of~\(G_2\) on \(A\rtimes G_1\) is
  defined by
  \[
  (\gamma'_{g_2} f)(g_1) \defeq \alpha_{g_2} (f(g_2^{-1}g_1g_2))
  \]
  for \(g_2\in G_2\), \(g_1\in G_1\), \(f\in\Contc(G_1, A)\) on the
  dense subalgebra \(\Contc(G_1,A)\); this extends to the
  \(\Cst\)\nb-completion.  The action~\(\acm_2\) of~\(G_2\) on~\(H_2\)
  leaves \(H_1\subseteq H_2=\ker \psi_2\) invariant because~\(\psi_2\)
  is \(G_2\)\nb-equivariant.  Since~\(\gamma'_g\) for \(g\in G_2\)
  maps \(u_h-\delta_{\tcm_1(h)}\) to
  \(u_{\acm_g(h)}-\delta_{\tcm_1(\acm_g(h))}\) and
  \(\acm_{G_2}(H_1)\subseteq H_1\), the action~\(\gamma'\) descends to
  an action~\(\gamma\) of~\(G_2\) on~\(A\rtimes\cm_1\).

  Let~\(U_g\) for \(g\in G_1\) be the image of
  \(\delta_g\in\U\Mult(A\rtimes G_1)\) in \(\U\Mult(A\rtimes\cm_1)\);
  this defines a homomorphism \(G_1\to\U\Mult(A\rtimes\cm_1)\) with
  \(\gamma_g = \Ad(U_g)\) all \(g\in G_1\).  Also let
  \(U_h\in\U\Mult(A\rtimes\cm_1)\) for \(h\in H_2\) be the image of
  \(u_h\in \U\Mult(A)\) under the canonical map \(A\to \Mult(A\rtimes
  G_1)\to\Mult(A\rtimes\cm_1)\).  We claim that \(\gamma_{\tcm_2(h)} =
  \Ad(U_h)\) for all \(h\in H_2\).  To see this, we notice first that
  \(\acm_{2,g}(h)h^{-1} \in H_1 = \ker \psi_2\) for \(g\in G_1\)
  because \(\varphi_2(g)=1\) implies \(\psi_2(\acm_{2,g}(h)h^{-1}) =
  \psi_2(\acm_{3,\varphi_2(g)}(h)h^{-1}) = 1\).  Hence
  \(U_{\acm_{2,g^{-1}}(h)h^{-1}} =
  \delta_{\tcm_1(\acm_{2,g^{-1}}(h)h^{-1})}\) holds in
  \(\U\Mult(A\rtimes\cm_1)\).  This implies
  \begin{multline*}
    U_h\delta_g U_h^*
    = \delta_g U_{\acm_{2,g^{-1}}(h)} U_{h^{-1}}
    = \delta_g U_{\acm_{2,g^{-1}}(h)h^{-1}}
    \\= \delta_g \delta_{\tcm_1(\acm_{2,g^{-1}}(h)h^{-1})}
    = \delta_{g \tcm_2(\acm_{2,g^{-1}}(h)) \tcm_2(h)^{-1})}
    = \delta_{\tcm_2(h)g\tcm_2(h)^{-1}}.
  \end{multline*}
  Since we have assumed \(u_h a u_h^* = \alpha_{\tcm_2(h)}(a)\) for
  all \(a\in A\), we get \(U_h x U_h^* = \gamma_{\tcm_2(h)}(x)\) for
  all \(x\in A\rtimes\cm_1\).

  If \(g\in G_1\), \(h\in H_2\), then \(U_g U_h U_g^*\) is the image
  of \(\alpha_g(u_h) = u_{\acm_{2,g}(h)}\), that is, \(U_g U_h U_g^* =
  U_{\acm_{2,g}(h)}\).  Thus the map \((g,h)\mapsto U_g U_h\) is a
  homomorphism \(G_1\ltimes H_2\to\U\Mult(A\rtimes\cm_1)\) where the
  semidirect product uses the action \(G_1\subseteq G_2\to\Aut(H_2)\)
  given by restricting~\(\acm_2\).  Since \(U_h=U_{\tcm_1(h)}\) for
  \(h_1\in H_1\subseteq H_2\), we get a homomorphism on \(H\defeq
  G_1\ltimes H_2/\Delta(H_1)\) with the embedding \(\Delta\colon
  H_1\to G_1\ltimes H_2\), \(h\mapsto (\tcm(h)^{-1},h)\).  The map
  \(G_1\ltimes H_2\to G_2\), \((g,h)\mapsto g\cdot\tcm_2(h)\), is a
  group homomorphism which vanishes on \(\Delta(H_1)\) and hence
  descends to a group homomorphism \(\tcm\colon H\to G_2\).  We define
  a homomorphism \(\acm'\colon G_2\to\Aut(G_1\ltimes H_2)\) by
  \(\acm'_{g_2}(g_1,h_2)\defeq (g_2g_1g_2^{-1},\acm_{2,g_2}(h_2))\).
  This leaves \(\Delta(H_1)\) invariant and hence descends to a
  homomorphism \(\acm\colon G_2\to\Aut(H)\).  Putting all this
  together gives a crossed module \(\cm\defeq (G_2,H,\tcm,\acm)\)
  which acts on \(A\rtimes \cm_1\) by \(\gamma\colon
  G_2\to\Aut(A\rtimes\cm_1)\) and \(U\colon
  H\to\U\Mult(A\rtimes\cm_1)\).

  The homomorphism~\(\tcm\) maps \(G_1\subseteq H\) homeomorphically
  onto the closed normal subgroup \(G_1\subseteq G_2\).  We have
  \(G_2/\tcm(G_1) \cong G_3\) and \(H/G_1\cong H_2/H_1\cong H_3\).
  Example~\ref{exa:cm-equivalence_quotient} shows that~\(\cm\) is
  equivalent to the crossed module~\(\cm_3\).  By
  Theorem~\ref{the:equivalence_for_actions}, the action \((\gamma,U)\)
  of~\(\cm\) on~\(A\rtimes\cm_1\) is equivalent to an action
  of~\(\cm_3\) on~\(A\rtimes\cm_1\) by correspondences such that
  \((A\rtimes\cm_1)\rtimes \cm \cong (A\rtimes\cm_1)\rtimes\cm_3\).
  We claim that \((A\rtimes\cm_1)\rtimes \cm \cong A\rtimes\cm_2\).

  By the universal property a morphism \((A\rtimes\cm_1)\rtimes\cm\to
  D\) is equivalent to a \(\cm\)\nb-covariant representation of
  \(A\rtimes\cm_1\) in \(\Mult(D)\), that is, a morphism
  \(\rho_C\colon A\rtimes\cm_1\to\Mult(D)\) and a continuous
  homomorphism \(V\colon G_2\to\U\Mult(D)\) with \(V_g\rho_C(c)V_g^* =
  \rho_C(\gamma_g(c))\) for all \(c\in A\rtimes\cm_1\), \(g\in G_2\)
  and \(V_{\tcm(h)}=\rho_C(U_h)\) for all \(h\in H\).  By the
  universal property of~\(A\rtimes\cm_1\), the
  representation~\(\rho_C\) is equivalent to a morphism \(\rho_A\colon
  A\to\U\Mult(D)\) and a continuous homomorphism \(W\colon
  G_1\to\U\Mult(D)\) with \(W_g\rho_A(a)W_g^* = \rho_A(\alpha_g(a))\)
  for all \(a\in A\), \(g\in G_1\) and \(W_{\tcm_1(h)}=\rho_A(u_h)\)
  for all \(h\in H_1\).  The assumptions on~\(\rho_C\) are equivalent
  to \(V_g\rho_A(a)V_g^* = \rho_A(\alpha_g(a))\) for \(a\in A\),
  \(g\in G_2\); \(V_{g_2}W_{g_1}V_{g_2}^* = W_{g_2g_1g_2^{-1}}\) for
  \(g_2\in G_2\), \(g_1\in G_1\); \(V_g=W_g\) for \(g\in G_1\); and
  \(V_{\tcm_2(h)} = \rho_A(u_h)\) for \(h\in H_2\).  Thus the
  unitaries~\(W_g\) for \(g\in G_1\) are redundant, and the conditions
  on~\(\rho_A\) and the unitaries~\(V_g\) for \(g\in G_2\) are
  precisely those for a covariant representation of \(A\)
  and~\(\cm_2\).  Hence the morphisms
  \((A\rtimes\cm_1)\rtimes\cm\to\Mult(D)\) are in natural bijection
  with morphisms \(A\rtimes\cm_2\to\Mult(D)\).  This shows that
  \(A\rtimes\cm_2\cong (A\rtimes\cm_1)\rtimes\cm\).  Since
  \((A\rtimes\cm_1)\rtimes \cm \cong (A\rtimes\cm_1)\rtimes\cm_3\),
  this gives the desired isomorphism.

  We must show that the \(\cm_3\)\nb-action on~\(A\rtimes\cm_1\) is
  natural, so that \(A\mapsto (A\rtimes\cm_1)\rtimes\cm_3\) is a
  functor.  Here we still talk about functors defined on the full
  sub-\(2\)\nb-category of \(\Corr(\cm_2)\) consisting of strict
  actions by automorphisms.  The naturality of the
  \(\cm_3\)\nb-action is equivalent to the naturality of the
  \(\cm\)\nb-action by Theorem~\ref{the:equivalence_for_actions},
  which is what we are going to prove.

  A \(\cm_2\)\nb-transformation between two strict actions on \(A_1\)
  and~\(A_2\) by automorphisms is equivalent to a \(G\)\nb-equivariant
  correspondence~\(\Hilm\) from~\(A_1\) to~\(A_2\) in the usual sense,
  subject to the extra requirement \(u_h^{A_1}\cdot\xi = \xi\cdot
  u_h^{A_2}\) for all \(h\in H_2\), \(\xi\in\Hilm\).  For such a
  correspondence, we get an action of~\(G_2\) on the induced
  correspondence \(\Hilm\rtimes\cm_1\) from \(A_1\rtimes\cm_1\) to
  \(A_2\rtimes\cm_2\) by the same formulas as above, and this yields a
  \(\cm\)\nb-equivariant correspondence from \(A_1\rtimes\cm_1\) to
  \(A_2\rtimes\cm_2\).  Furthermore, isomorphic
  \(\cm_2\)\nb-equivariant correspondences induce isomorphic
  \(\cm\)\nb-equivariant correspondences.  Hence the \(\cm\)\nb-action
  on~\(A\rtimes\cm_1\) is natural on the \(2\)\nb-category of strict
  actions of~\(\cm_2\) by automorphisms.

  The isomorphism \(A\rtimes\cm_2 \to (A\rtimes\cm_1)\rtimes\cm\) is
  natural in the sense that for any \(\cm_2\)\nb-equivariant
  correspondence~\(\Hilm\) from \(A_1\) to~\(A_2\), the square formed
  by the isomorphisms above and the induced correspondences
  \(A_1\rtimes\cm_2\to A_2\rtimes\cm_2\) and
  \((A_1\rtimes\cm_1)\rtimes\cm_3\to(A_2\rtimes\cm_1)\rtimes\cm_3\)
  commutes up to a canonical isomorphism of correspondences.  This
  establishes the naturality of our isomorphism on the \(2\)\nb-category of
  strict actions of \(\cm_2\).

  By the Packer--Raeburn Stabilisation Trick, any action of~\(\cm_2\)
  by correspondences is equivalent to a strict \(\cm_2\)\nb-action by
  automorphisms (\cite{Buss-Meyer-Zhu:Higher_twisted}*{Theorem 5.3}),
  where equivalence means an isomorphism (that is, equivariant Morita equivalence) in the \(2\)\nb-category
  \(\Corr(\cm_2)\).  This equivalence means that a functor defined
  only on the subcategory of strict \(\cm_2\)\nb-actions may be
  extended to a functor on all of \(\Corr(\cm_2)\); all such
  extensions are naturally isomorphic; and a natural transformation
  between functors on the subcategory extends to a natural
  transformation between the extensions.  Hence the result for strict
  actions proves the more general result for actions by
  correspondences by abstract nonsense.
\end{proof}

\begin{example}
Let \(G\) be a locally compact group and let \(N\) be a closed normal subgroup of \(G\) so that we get a group extension \(N\into G\onto G/N\).
Viewing \(N\), \(G\) and \(G/N\) as crossed modules \(\cm_1=(N,0,0,0)\), \(\cm_2=(G,0,0,0)\) and \(\cm_3=(G/N,0,0,0)\), respectively, our result says that
given an action \(\alpha\) of \(G\) on a \cstar{}algebra \(A\) by correspondences, there is an action \(\beta\) of \(G/N\) on \(A\rtimes_{\alpha|} N\) by correspondences, where \(\alpha|\) denotes the restriction of \(\alpha\) to \(N\), such that
\[
A\rtimes_\alpha G\cong (A\rtimes_{\alpha|} N)\rtimes_{\beta} G/N.
\]
We may also interpret everything in terms of Fell bundles: the
action of~\(G\) on~\(A\) corresponds to a Fell bundle~\(\A\)
over~\(G\) with unit fibre \(\A_1=A\) in such way that
\(A\rtimes_\alpha G\) is (isomorphic to) the cross-sectional
\cstar{}algebra~\(\Cst(\A)\). The restricted crossed product
\(A\rtimes_{\alpha|}N\) corresponds to the cross-sectional
\cstar{}algebra~\(\Cst(\A_N)\) of the restriction~\(\A_N\) of~\(\A\)
to~\(N\).  Our theorem says that there is a Fell bundle~\(\B\)
over~\(G/N\) with unit fibre \(\B_1=A\rtimes_{\alpha|} N\cong
C^*(\A|_N)\) such that \(C^*(\B)\cong C^*(\A)\).  Although our
\cstar{}algebraic version appears to be new, a version for
\(L^1\)-cross-sectional algebras in proved by Doran and Fell in
\cite{Doran-Fell:Representations_2}*{VIII.6}.

Even if we start with a strict action of~\(G\) on~\(A\) by
\emph{automorphisms}, the induced action of~\(G/N\) on~\(A\rtimes
N\) will usually not be an action by automorphisms.  It may be
interpreted as a Green twisted action of \((G,N)\) on \(A\rtimes
N\), and the above decomposition corresponds to Green's
decomposition of crossed products: \(A\rtimes_\alpha G\cong
(A\rtimes N)\rtimes(G,N)\) (see \cites{Green:Local_twisted,
  Echterhoff:prime_ideal_space}).
\end{example}

We may weaken the notion of strict extension by replacing the
crossed modules involved by equivalent ones.  We mention only one
relevant example of this.

\begin{example}
  \label{exa:cm_as_extension}
  Let \(\cm=(G,H,\tcm,\acm)\) be a crossed module.  Let \(G_2\defeq
  G\ltimes_{\acm} H\) be the semidirect product group.  It contains
  \(H_2\defeq H\) as a normal subgroup via \(\tcm_2\colon H_2\to
  G_2\), \(h\mapsto (1,h)\), with quotient \(G_2/H_2\cong G\).  Let
  \(\acm_2\colon G_2\to\Aut(H_2)\) be the resulting conjugation
  action, \(\acm_{2,(g,h)}(k)\defeq \acm_g(hkh^{-1})\).  Then
  \(\cm_2=(G_2,H_2,\tcm_2,\acm_2)\) is a crossed module of locally
  compact groups that is equivalent to \((G,0,0,0)\).  Hence actions
  of~\(\cm_2\) are equivalent to actions of the group~\(G\), with
  the same crossed products on both sides
  (Theorem~\ref{the:equivalence_for_actions}).  Let
  \(\cm_1=(H,0,0,0)\) be the group~\(H\) turned into a crossed
  module and let \(\cm_3=\cm\).  We map \(\cm_2\to\cm\) by
  \(\psi_2=\Id\colon H\to H\) and \(\varphi_2\colon G\ltimes_{\acm}
  H\to H\), \((g,h)\mapsto g\cdot\tcm(h)\).  This is a homomorphism
  of crossed modules, and \(\psi_2\) and~\(\varphi_2\) are open
  surjections.  Their kernels are isomorphic to \(H_1\defeq 0\) and
  \(G_1\defeq H\) via \(\varphi_1\colon H\to G\ltimes_{\acm} H\),
  \(h\mapsto (\tcm(h)^{-1},h)\), respectively.  Thus we get a strict
  extension of crossed modules \(\cm_1\to \cm_2\to \cm\) with
  \(\cm_1=(H,0,0,0)\) and \(\cm_2\simeq (G,0,0,0)\).  Hence the
  group~\(G\) is equivalent to an extension of the group~\(H\) by
  the crossed module~\(\cm\).

  Now let~\(A\) carry an action of~\(G\), which we turn into an
  action of~\(\cm_2\) via \(G\ltimes H\to G\), \((g,h)\mapsto
  g\tcm(h)\).  When we apply Theorem~\ref{the:partial_crossed} to
  this situation, we get back
  \cite{Buss-Meyer-Zhu:Non-Hausdorff_symmetries}*{Theorem 1}:
  \[
  A\rtimes G\cong (A\rtimes H)\rtimes \cm.
  \]
\end{example}

\begin{example}
  Let~\(\theta\) be some irrational number and define an embedding
  \(\theta\colon \Z\to \R\) by \(n\mapsto \theta n\).  Let
  \(\cm=(\R,\Z,\theta,0)\) be the resulting Abelian crossed module.
  This is equivalent to the group~\(\Torus\) (viewed as the crossed
  module \((\Torus,0,0,0)\)) via the homomorphism
  \((\varphi,\psi)\colon\cm \to\Torus\) with \(\varphi\colon \R\to
  \Torus\), \(t\mapsto \exp(2\pi \ima\theta t)\), and the trivial
  homomorphism \(\psi\colon\Z\to 0\)
  (Example~\ref{exa:grouplike_classical}).
  Theorem~\ref{the:equivalence_for_actions} gives an equivalence of
  \(2\)\nb-categories \(\Corr(\Torus)\congto \Corr(\cm)\); it sends a
  \(\Torus\)\nb-algebra~\(A\) to itself with~\(\Z\) acting trivially
  and~\(\R\) via~\(\varphi\) and the given \(\Torus\)\nb-action.

  Now let \(\cm'=(\Torus,\Z,\tcm,0)\) be the crossed module
  considered in Example~\ref{exa:crossed_Z_T}, where
  \(\tcm(n)=\exp(2\pi\ima\theta n)\).  View the group~\(\Z\) as a
  crossed module.  There is an extension \(\Z\into \cm\onto\cm'\)
  described by the diagram:
  \[
  \begin{tikzpicture}[yscale=1.2,xscale=2,baseline=(current bounding box.west)]
    \node (m-0-1) at (0,2) {0};
    \node (m-0-2) at (1,2) {\(\Z\)};
    \node (m-0-3) at (2,2) {\(\Z\)};
    \node (m-1-1) at (0,1) {\(\Z\)};
    \node (m-1-2) at (1,1) {\(\R\)};
    \node (m-1-3) at (2,1) {\(\Torus\)};
    \draw[cdar] (m-0-1) -- node {} (m-0-2);
    \draw[cdar] (m-0-2) -- node {\Id} (m-0-3);
    \draw[cdar] (m-0-1) -- node[swap] {} (m-1-1);
    \draw[cdar] (m-0-3) -- node {\(\tcm\)} (m-1-3);
    \draw[cdar] (m-0-2) -- node[swap] {\(\theta\)} (m-1-2);
    \draw[cdar] (m-1-1) -- node {\Id} (m-1-2);
    \draw[cdar] (m-1-2) -- node {\(\varphi\)} (m-1-3);
  \end{tikzpicture}
  \]
  Therefore, Theorem~\ref{the:partial_crossed} gives a functor
  \(\Corr(\cm)\to\Corr(\cm')\) that sends a \(\cm\)\nb-algebra~\(A\)
  to the (restricted) crossed product \(A\rtimes\Z\) with an induced
  \(\cm'\)\nb-action, such that \(A\rtimes\cm\cong
  (A\rtimes\Z)\rtimes\cm'\).

  Composing this with the equivalence \(\Corr(\Torus)\cong\Corr(\cm)\)
  we obtain a functor \(\Corr(\Torus)\to\Corr(\cm')\) that sends a
  \(\Torus\)\nb-algebra~\(A\) to \(A\rtimes\Z\) with a
  \(\cm'\)\nb-action such that \((A\rtimes\Z)\rtimes\cm'\cong
  A\rtimes\Torus\).  As a simple example, we take the
  \(\Torus\)\nb-algebra \(\Cont(\Torus)\) with translation
  \(\Torus\)\nb-action. In this case, \(\Z\) acts by irrational
  rotation by multiples of~\(\theta\) so that
  \(\Cont(\Torus)\rtimes\Z\cong C^*(\Torus_\theta)\) is the
  noncommutative torus and the induced \(\cm'\)\nb-action is the same
  one considered in Example~\ref{exa:crossed_Z_T}.  Hence we get once
  again that \(C^*(\Torus_\theta)\rtimes\cm'\cong
  \Cont(\Torus)\rtimes\Torus\cong \Comp(L^2\Torus)\).

  Theorem~\ref{the:Abelian_cm_groupoid} shows that \(\Corr(\cm)\cong
  \Corr(\hat{\R}\ltimes\hat{\Z})\cong \Corr(\R\ltimes\Torus)\),
  where \(\R\ltimes \Torus\) denotes the transformation groupoid for
  the action \(t\cdot z=\exp(2\pi \ima\theta t)z\) for \(t\in \R\)
  and \(z\in\Torus\).  Composing \(\Corr(\Torus)\congto \Corr(\cm)\)
  with this equivalence, we get a functor \(\Corr(\Torus)\congto
  \Corr(\R\ltimes\Torus)\).  This takes a \(\Torus\)\nb-algebra,
  views it as a \(\cm\)\nb-algebra, and sends it to the crossed
  product \(A\rtimes\R\) viewed as an \(\R\ltimes\Torus\)-algebra
  using the dual \(\R\)\nb-action and the structure of
  \(\Cont(\Torus)\)-algebra given by the homomorphism
  \(\Cont(\Torus)\cong \Cst(\Z)\to\Cst(\R) \to\Mult(A\rtimes\R)\),
  which maps \(\Cont(\Torus)\) into the centre of~\(\Mult(A)\).

  Theorem~\ref{the:Abelian_cm_groupoid} also gives
  \(\Corr(\cm')\cong \Corr(\hat{\Torus}\ltimes\hat{\Z}) \cong
  \Corr(\Z\ltimes\Torus)\), where~\(\Z\) acts by rotation by
  multiples of~\(\theta\).  The quotient map \(\cm\to\cm'\) becomes
  the forgetful functor that restricts an \(\R\ltimes\Torus\)-action
  to a \(\Z\ltimes\Torus\)-action on \(\theta\Z\subseteq\R\).
\end{example}

\section{Factorisation of the crossed product functor}
\label{sec:factorise_crossed_product}

Now we put our results together to factorise the crossed product
functor \(\Corr(\cm)\to\Corr\) for a crossed module
\(\cm=(G,H,\tcm,\acm)\) of locally compact groups into ``elementary''
constructions.  First we give more details on the strict extensions
that decompose~\(\cm\) into simpler building blocks.

The image \(\tcm(H)\) is a normal subgroup in~\(G\) because
\(g\tcm(h)g^{-1}=\tcm_{\acm_g(h)}\).  Hence \(\cl{\tcm(H)}\) is a
closed normal subgroup in~\(G\) and
\[
\bar{\pi}_1(\cm) \defeq G/\cl{\tcm(H)}
\]
is a locally compact group.  The closed subgroup
\[
\pi_2(\cm) \defeq \ker \tcm \subseteq H
\]
is Abelian because \(hkh^{-1}=\acm_{\tcm(h)}(k)\) for all \(h,k\in
H\).

Since \(\pi_2(\cm)\) is an Abelian locally compact group, there is a
crossed module~\(\cm_1\) with \(H_1=\pi_2(\cm)\) and trivial~\(G_1\)
(and hence trivial \(\tcm_1\) and~\(\acm_1\)).  The
\(G\)\nb-action~\(\acm\) on~\(H\) leaves~\(\pi_2(\cm)\) invariant
and hence descends to an action~\(\acm_2\) of~\(G\) on \(H_2\defeq
H/\pi_2(\cm)\).  Of course, \(\tcm\) descends to a map
\(\tcm_2\colon H_2\to G_2=G\).  This defines a crossed module of
locally compact groups~\(\cm_2\).  The canonical maps
\(\cm_1\to\cm\to\cm_2\) are homomorphisms of crossed modules, and
they clearly form a strict extension of crossed modules, based on
the extensions of locally compact groups \(\pi_2(\cm)\into H\onto
H/\pi_2(\cm)\) and \(0\into G=G\).

There is a crossed module~\(\cm_3\) with \(G_3\defeq \cl{\tcm(H)}\),
\(H_3=H_2=H/\pi_2(\cm)\), and \(\tcm_3\colon H_3\to G_3\) and
\(\acm_3\colon G_3\to\Aut(H_3)\) induced by \(\tcm\) and~\(\acm\).
Let~\(\cm_4\) be the crossed module with \(G_4=\bar{\pi}_1(\cm)\)
and trivial \(H_4\), \(\tcm_4\) and~\(\acm_4\).  The obvious maps
give homomorphisms of crossed modules \(\cm_3\to \cm_2\to \cm_4\);
these form a strict extension of crossed modules because we have
extensions \(H_3=H_2\onto 0\) and \(G_3\into G_2\onto G_4\) of
locally compact groups.

The strict extensions above show that the crossed product functor
for \(\cm\)\nb-actions factorises into the three crossed
product functors with \(\cm_1\), \(\cm_3\) and~\(\cm_4\).  Now we
analyse actions and crossed products for \(\cm_1\), \(\cm_4\),
and~\(\cm_3\), respectively.

Since a crossed module~\(\cm_1\) of the form \((0,H_1,0,0)\) is
Abelian, Theorem~\ref{the:Abelian_cm_groupoid} shows that
\(\Corr(\cm_1)\) is equivalent to the \(2\)\nb-category of
\(\Cont_0(\widehat{H_1})\)-\(\Cst\)-algebras, such that the crossed
product by~\(\cm_1\) corresponds to the functor that maps a
\(\Cont_0(\widehat{H_1})\)-\(\Cst\)-algebra to its fibre at
\(1\in\widehat{H_1}\).  The case at hand is much easier than the
general case of Theorem~\ref{the:Abelian_cm_groupoid} because~\(G\) is
trivial.  We simply observe that a \(\cm_1\)\nb-action on~\(A\) is
exactly the same as a nondegenerate \Star{}homomorphism from
\(\Cont_0(\hat{H})\cong\Cst(H)\) to the central multiplier algebra
of~\(A\).

For crossed modules of the form \(\cm_4=(G_4,0,0,0)\), there is
nothing to analyse: actions of this crossed module are the same as
actions of the locally compact group~\(G_4\), and the crossed
product functor is also the same as for group actions.  We already
showed in~\cite{Buss-Meyer-Zhu:Higher_twisted} that group actions by
correspondences are equivalent to saturated Fell bundles.  The
crossed product is the cross-sectional \(\Cst\)\nb-algebra of a Fell
bundle.  By the Packer--Raeburn Stabilisation Trick (see also
\cite{Buss-Meyer-Zhu:Higher_twisted}*{Theorem 5.3}), we may
replace \(G_4\)\nb-actions by correspondences by ordinary continuous
group actions on a stabilisation.  This replaces the crossed product
functor for actions of~\(\cm_4\) by a classical crossed product
construction for actions of the locally compact group~\(G_4\).

Now we study crossed products by the thin crossed module
\[
\cm_3=(\cl{\tcm{H}},H/\ker\tcm,\tcm_3,\acm_3),
\]
using the results in Section~\ref{sec:simplification_thin} to
replace~\(\cm_3\) by an equivalent Abelian crossed
module~\(\cm_5\).  By Theorem~\ref{the:thin_Lie_equivalent_Abelian},
such an Abelian model for~\(\cm_3\) exists if both \(G\)
and~\(H\) are Lie groups.  More generally,
Lemma~\ref{lem:thin_cm_commutative_sufficient} gives an Abelian
model if \(H/\ker \tcm\) has a compactly generated subgroup~\(A\)
for which~\(\cl{\tcm(A)}\) is open in~\(\cl{\tcm(H)}\); in
particular, this happens if~\(H\) itself is compactly generated.
Theorem~\ref{the:thin_locally_compact_normalise} also gives a
necessary and sufficient condition for an Abelian model to exist;
but this criterion does not explain why this happens so often.

Assume that~\(\cm_3\) is equivalent to an Abelian crossed module
\(\cm_5=(G_5,H_5,\tcm_5,\acm_5)\); even better, we can achieve
that~\(G_5\) is compact Abelian, \(H_5\) is discrete, and~\(\acm_5\)
is trivial.  Since~\(\cm_3\) is thin, so is~\(\cm_5\), that is,
\(\tcm_5\) is an injective map with dense range.  By
Theorem~\ref{the:equivalence_for_actions}, the \(2\)\nb-categories
\(\Corr(\cm_3)\) and \(\Corr(\cm_5)\) of actions of \(\cm_3\)
and~\(\cm_5\) by correspondences are equivalent, in such a way that
the crossed product functors on both categories are identified.
Moreover, the proof shows immediately that the underlying
\(\Cst\)\nb-algebra is not changed: an action of~\(\cm_3\) becomes a
\(\cm_5\)\nb-action on the same \(\Cst\)\nb-algebra.  For crossed
modules of Lie groups, we have explained in
Section~\ref{sec:equivalence_crossed_modules} how to
construct~\(\cm_5\) explicitly out of~\(\cm_3\).

Theorem~\ref{the:Abelian_cm_groupoid} shows that \(\Corr(\cm_5)\) is
equivalent to the \(2\)\nb-category of actions of the groupoid
\(\widehat{G_5}\ltimes \widehat{H_5}\); this equivalence maps a
\(\cm_5\)\nb-action to the crossed product by~\(G_5\) equipped with
a canonical \(\Cont_0(\widehat{H_5})\)-\(\Cst\)-algebra structure
and the dual action of~\(\widehat{G_5}\); the crossed product
by~\(\cm_5\) corresponds to taking the fibre at~\(1\) for the
\(\Cont_0(\widehat{H_5})\)-\(\Cst\)-algebra structure.  Thus after
an equivalence
\[
\Corr(\cm_3)\simeq \Corr(\cm_5)
\simeq \Corr(\widehat{G_5}\ltimes \widehat{H_5})
\]
that on the underlying \(\Cst\)\nb-algebras takes a crossed product
with the Abelian compact group~\(G_5\), the crossed product
with~\(\cm_3\) becomes a fibre restriction functor.

The following theorem summarises our factorisation of the crossed
product:

\begin{theorem}
  \label{the:decompose_crossed}
  Let~\(\cm\) be a crossed module of Lie groups or, more generally, a
  crossed module of locally compact groups for which the associated
  thin crossed module \(H/\ker \tcm \to \cl{\tcm(H)}\) is equivalent
  to an Abelian crossed module.  There are a locally compact Abelian
  group~\(X\), compact Abelian groups \(Y\) and~\(K\), and a locally
  compact group~\(L\), such that for any action of~\(\cm\) by
  correspondences on a \(\Cst\)\nb-algebra~\(A\),
  \begin{enumerate}
  \item \(A\) carries a natural \(\Cont_0(X)\)-\(\Cst\)-algebra structure;
  \item the unit fibre~\(A_1\) of \(A\) for this natural
    \(\Cont_0(X)\)-\(\Cst\)-algebra structure carries a natural action
    of~\(K\) by correspondences;
  \item the crossed product \(A_2\defeq A_1\rtimes K\) carries a
    natural \(\Cont_0(Y)\)-\(\Cst\)-algebra structure;
  \item the unit fibre~\(A_3\) of~\(A_2\) for this natural
    \(\Cont_0(Y)\)-\(\Cst\)-algebra structure carries a natural action
    of~\(L\) by correspondences;
  \item the crossed product \(A_3\rtimes L\) is naturally
    isomorphic to~\(A\rtimes\cm\).
  \end{enumerate}
\end{theorem}

\subsection{Computing K-theory of crossed module crossed products}
\label{sec:K-theory}

In the localisation formulation of~\cite{Meyer-Nest:BC}, the
Baum--Connes assembly map for a locally compact group~\(G\) compares
the \(\K\)\nb-theory of the reduced crossed product with a more
topological invariant that uses only crossed products for restrictions
of the action to compact subgroups of~\(G\).  Its assertion is
therefore trivial if~\(G\) is itself compact.  Crossed products for
compact groups are an ``elementary'' operation for \(\K\)\nb-theory
purposes in the sense that there is no better way to compute the
\(\K\)\nb-theory than the direct one.  Crossed products for
non-compact groups are not ``elementary'' in this sense because the
Baum--Connes conjecture (if true) allows us to reduce the
\(\K\)\nb-theory computation to \(\K\)\nb-theory computations for
compact subgroups and some algebraic topology to assemble the results
of these computations.

Taking the fibre in a \(\Cont_0(X)\)-\(\Cst\)\nb-algebra seems to be
an operation that is also ``elementary'' in the above sense.  At
least, we know of no better way to compute the \(\K\)\nb-theory of a
fibre than the direct one.  Notice that
\(\Cont_0(X)\)-\(\Cst\)\nb-algebras need not be locally trivial.

In the notation of Theorem~\ref{the:decompose_crossed}, the functors
\(A\mapsto A_1\mapsto A_2\mapsto A_3\) are therefore ``elementary''
for \(\K\)\nb-theory purposes.  The remaining fourth step \(A_3\mapsto
A_3\rtimes L\) is the (full) crossed product by the locally compact
group \(L=G/\cl{\tcm(H)}\).  Many results are available about the
\(\K\)\nb-theory of such crossed products.

Hence our decomposition of crossed module crossed products also
gives us a useful recipe for computing their \(\K\)\nb-theory.  This
recipe is, however, quite different from the localisation approach
for groups in~\cite{Meyer-Nest:BC}.


\begin{bibdiv}
  \begin{biblist}
\bib{Buss-Meyer-Zhu:Non-Hausdorff_symmetries}{article}{
  author={Buss, Alcides},
  author={Meyer, Ralf},
  author={Zhu, {Ch}enchang},
  title={Non-Hausdorff symmetries of \(\textup C^*\)\nobreakdash -algebras},
  journal={Math. Ann.},
  issn={0025-5831},
  volume={352},
  number={1},
  pages={73--97},
  date={2012},
  review={\MRref {2885576}{}},
  doi={10.1007/s00208-010-0630-3},
}

\bib{Buss-Meyer-Zhu:Higher_twisted}{article}{
  author={Buss, Alcides},
  author={Meyer, Ralf},
  author={Zhu, {Ch}enchang},
  title={A higher category approach to twisted actions on \(\textup C^*\)\nobreakdash -algebras},
  journal={Proc. Edinb. Math. Soc. (2)},
  date={2013},
  volume={56},
  number={2},
  pages={387--426},
  issn={0013-0915},
  doi={10.1017/S0013091512000259},
  review={\MRref {3056650}{}},
}

\bib{Doran-Fell:Representations_2}{book}{
  author={Doran, Robert S.},
  author={Fell, James M. G.},
  title={Representations of $^*$\nobreakdash -algebras, locally compact groups, and Banach $^*$\nobreakdash -algebraic bundles. Vol. 2},
  series={Pure and Applied Mathematics},
  volume={126},
  publisher={Academic Press Inc.},
  place={Boston, MA},
  date={1988},
  pages={i--viii and 747--1486},
  isbn={0-12-252722-4},
  review={\MRref {936629}{90c:46002}},
}

\bib{Echterhoff:prime_ideal_space}{article}{
  author={Echterhoff, Siegfried},
  title={The primitive ideal space of twisted covariant systems with continuously varying stabilizers},
  journal={Math. Ann.},
  volume={292},
  date={1992},
  number={1},
  pages={59--84},
  issn={0025-5831},
  doi={10.1007/BF01444609},
  review={\MRref {1141785}{93a:46128}},
}

\bib{Echterhoff-Kaliszewski-Quigg-Raeburn:Categorical}{article}{
  author={Echterhoff, Siegfried},
  author={Kaliszewski, Steven P.},
  author={Quigg, John},
  author={Raeburn, Iain},
  title={A categorical approach to imprimitivity theorems for $C^*$\nobreakdash -dynamical systems},
  journal={Mem. Amer. Math. Soc.},
  volume={180},
  date={2006},
  number={850},
  pages={viii+169},
  issn={0065-9266},
  review={\MRref {2203930}{2007m:46107}},
  doi={10.1090/memo/0850},
}

\bib{Echterhoff-Quigg:InducedCoactions}{article}{
  author={Echterhoff, Siegfried},
  author={Quigg, John},
  title={Induced coactions of discrete groups on $C^*$\nobreakdash -algebras},
  journal={Canad. J. Math.},
  volume={51},
  date={1999},
  number={4},
  pages={745--770},
  issn={0008-414X},
  review={\MRref {1701340}{2000k:46094}},
  doi={10.4153/CJM-1999-032-1},
}

\bib{Green:Local_twisted}{article}{
  author={Green, Philip},
  title={The local structure of twisted covariance algebras},
  journal={Acta Math.},
  volume={140},
  date={1978},
  number={3-4},
  pages={191--250},
  issn={0001-5962},
  review={\MRref {0493349}{58\,\#12376}},
  doi={10.1007/BF02392308},
}

\bib{Hilsum-Skandalis:Morphismes}{article}{
  author={Hilsum, Michel},
  author={Skandalis, Georges},
  title={Morphismes \(K\)\nobreakdash -orient\'es d'espaces de feuilles et fonctorialit\'e en th\'eorie de Kasparov \textup (d'apr\`es une conjecture d'A. Connes\textup )},
  journal={Ann. Sci. \'Ecole Norm. Sup. (4)},
  volume={20},
  date={1987},
  number={3},
  pages={325--390},
  issn={0012-9593},
  review={\MRref {925720}{90a:58169}},
  eprint={http://www.numdam.org/item?id=ASENS_1987_4_20_3_325_0},
}

\bib{Leinster:Basic_Bicategories}{article}{
  author={Leinster, Tom},
  title={Basic Bicategories},
  date={1998},
  status={eprint},
  note={\arxiv {math/9810017}},
}

\bib{MacLane-Whitehead:3-type}{article}{
  author={MacLane, Saunders},
  author={Whitehead, John Henry Constantine},
  title={On the $3$\nobreakdash -type of a complex},
  journal={Proc. Nat. Acad. Sci. U. S. A.},
  volume={36},
  date={1950},
  pages={41--48},
  review={\MRref {0033519}{11,450h}},
  eprint={http://www.pnas.org/content/36/1/41.full.pdf+html},
}

\bib{Meyer-Nest:BC}{article}{
  author={Meyer, Ralf},
  author={Nest, Ryszard},
  title={The Baum--Connes conjecture via localisation of categories},
  journal={Topology},
  volume={45},
  date={2006},
  number={2},
  pages={209--259},
  issn={0040-9383},
  review={\MRref {2193334}{2006k:19013}},
  doi={10.1016/j.top.2005.07.001},
}

\bib{Muhly:BundlesGroupoids}{article}{
  author={Muhly, Paul S.},
  title={Bundles over groupoids},
  conference={ title={Groupoids in analysis, geometry, and physics}, address={Boulder, CO}, date={1999}, },
  book={ series={Contemp. Math.}, volume={282}, publisher={Amer. Math. Soc.}, place={Providence, RI}, },
  date={2001},
  pages={67--82},
  review={\MRref {1855243}{2003a:46085}},
  doi={10.1090/conm/282/04679},
}

\bib{Muhly-Renault-Williams:Equivalence}{article}{
  author={Muhly, Paul S.},
  author={Renault, Jean N.},
  author={Williams, Dana P.},
  title={Equivalence and isomorphism for groupoid \(C^*\)\nobreakdash -algebras},
  journal={J. Operator Theory},
  volume={17},
  date={1987},
  number={1},
  pages={3--22},
  issn={0379-4024},
  review={\MRref {873460}{88h:46123}},
  eprint={http://www.theta.ro/jot/archive/1987-017-001/1987-017-001-001.pdf},
}

\bib{Muhly-Williams:Equivalence.FellBundles}{article}{
  author={Muhly, Paul S.},
  author={Williams, Dana P.},
  title={Equivalence and disintegration theorems for Fell bundles and their \(C^*\)\nobreakdash -algebras},
  journal={Dissertationes Math. (Rozprawy Mat.)},
  volume={456},
  date={2008},
  pages={1--57},
  issn={0012-3862},
  review={\MRref {2446021}{2010b:46146}},
  doi={10.4064/dm456-0-1},
}

\bib{Nilsen:Bundles}{article}{
  author={Nilsen, May},
  title={\(C^*\)\nobreakdash -bundles and \(C_0(X)\)-algebras},
  journal={Indiana Univ. Math. J.},
  volume={45},
  date={1996},
  number={2},
  pages={463--477},
  issn={0022-2518},
  review={\MRref {1414338}{98e:46075}},
  doi={10.1512/iumj.1996.45.1086},
}

\bib{Noohi:two-groupoids}{article}{
  author={Noohi, Behrang},
  title={Notes on 2\nobreakdash -groupoids, 2\nobreakdash -groups and crossed modules},
  journal={Homology, Homotopy Appl.},
  volume={9},
  date={2007},
  number={1},
  pages={75--106},
  issn={1532-0073},
  review={\MRref {2280287}{2007m:18006}},
  eprint={http://projecteuclid.org/euclid.hha/1175791088},
}

\bib{Yamagami:IdealStructure}{article}{
  author={Yamagami, Shigeru},
  title={On primitive ideal spaces of $C^*$\nobreakdash -algebras over certain locally compact groupoids},
  booktitle={Mappings of operator algebras (Philadelphia, PA, 1988)},
  series={Progr. Math.},
  volume={84},
  pages={199--204},
  publisher={Birkh\"auser Boston},
  place={Boston, MA},
  date={1990},
  review={\MRref {1103378}{92j:46110}},
}
  \end{biblist}
\end{bibdiv}
\end{document}